\documentclass[11pt]{amsart}
\usepackage{amsmath, amsthm, amssymb}
\usepackage[all]{xy}
\usepackage{mathtools}
\usepackage{enumitem}
\usepackage{mathrsfs}
\usepackage{tikz}
\usepackage{graphicx}
\usepackage[colorlinks=false]{hyperref}
\usepackage{amsxtra}

\newcommand\restr[2]{\ensuremath{\left.#1\right|_{#2}}}
\newcommand{\G}{\mathcal{G}}
\newcommand{\Gt}{\mathcal{G}_{{\rm{tight}}}}
\newcommand{\mcS}{\mathcal{S}}
\newcommand{\mcH}{\mathrel{\mathcal{H}}}
\newcommand{\mcmu}{\mathrel{\mu}}
\newcommand{\mcJ}{\mathrel{\mathcal{J}}}

\newcommand{\mcrho}{\mathrel{\rho}}
\newcommand{\mcV}{\mathcal{V}}
\newcommand{\mcI}{\mathcal{I}}

\newcommand{\bbC}{\mathbb{C}}

\newcommand{\iso}{\operatorname{Iso}}
\newcommand{\fE}{\widehat{E}_0}
\newcommand{\tfE}{\widehat{E}_{{\rm{tight}}}}
\newcommand{\up}{\uparrow\!}
\newcommand{\down}{\downarrow\!}
\newcommand{\lr}{\mathrel{\leftrightarrow}}

\newcommand{\covby}{\longrightarrow}

\newcommand{\supp}{\operatorname{supp}}

\newcommand{\htt}{h_{{\rm{tight}}}}

\newtheorem{prop}{Proposition}[section]
\newtheorem{thm}[prop]{Theorem}
\newtheorem{cor}[prop]{Corollary}
\newtheorem{lem}[prop]{Lemma}
\theoremstyle{definition}
\newtheorem{defn}[prop]{Definition}

\newtheorem{exmp}[prop]{Example}
\newtheorem{rem}[prop]{Remark}

\newlist{thmenum}{enumerate}{10}
\setlist[thmenum,1]{label=\textnormal{(\alph*)}}
\setlist[thmenum,2]{label=\textnormal{(\roman*)}}

\allowdisplaybreaks

\begin{document}

\title[Condition (K) and Ideal Structure]{Condition (K) for Inverse Semigroups and the ideal structure of their $C^*$-algebras}
\author{Scott M. LaLonde \and David Milan}
\address{Department of Mathematics\\
The University of Texas at Tyler\\
3900 University Boulevard\\
Tyler, TX 75799}
\email{slalonde@uttyler.edu, dmilan@uttyler.edu}

\author{Jamie Scott}
\address{Department of Mathematics \\
1400 Stadium Road \\
University of Florida \\
Gainesville, FL 32611}
\email{hiwhoareyou1@ufl.edu}
\thanks{The third author was supported by an NSF grant (DMS-1359101).}

\date{\today}


\maketitle

\begin{abstract}
	Inspired by results for graph $C^*$-algebras, we investigate connections between the ideal structure of an inverse semigroup $S$ and that of its tight $C^*$-algebra by
	relating ideals in $S$ to certain open invariant sets in the associated tight groupoid. We also develop analogues of Conditions (L) and (K) for inverse semigroups, which
	are related to certain congruences on $S$. We finish with applications to the inverse semigroups of self-similar graph actions and some relevant comments 
	on the authors' earlier uniqueness theorems for inverse semigroups.
\end{abstract}

\section{Introduction}

The theory of graph $C^*$-algebras has inspired a number of generalizations, and consequently it has provided motivation for a large body of research in $C^*$-algebras in 
recent years. In some ways this paper is a follow up to \cite{LaLondeMilan}, another paper inspired by graph $C^*$-algebras results, in which two uniqueness theorems for 
$C^*$-algebras of inverse semigroups were proved using the generalized uniqueness theorem for \'{e}tale groupoids from \cite{BNRSW}. If one is to follow the program laid out 
in the graph $C^*$-algebra literature, the natural next step is to describe the ideal structure of the tight $C^*$-algebra of an inverse semigroup. 

For directed graphs, there are two important conditions used to determine the ideal structure of the graph $C^*$-algebra. The first, Condition (L), says that every cycle in the 
graph has an entrance. This condition is equivalent to the property that the path groupoid of the graph is effective. The second, Condition (K), is that every vertex that is the 
base of a cycle has at least two distinct return paths. This condition is equivalent to the property that the reduction of the path groupoid to any closed invariant subset of its 
unit space is effective \cite[Corollary 6.5]{CE-MHS}. That is, Condition (K) is equivalent to the property that the path groupoid is \emph{strongly effective}.

In this paper, we develop the appropriate notions of Conditions (L) and (K) for inverse semigroups. For an inverse semigroup $S$, we will say that 
\[
	S \text{ satisfies Condition (L) if and only if } S/\!\!\lr \text{ is fundamental},
\]
where $\,\lr$ is a congruence that was introduced in \cite{lenz} that identifies any two idempotents that must be identified under any tight representation of $S$. Extending a characterization of Exel and Pardo 
\cite{ExelPardo}, we show under a suitable condition on $S$ that the tight groupoid of $S$ is effective if and only if $S/\!\!\lr$ is fundamental.

We will say that $S$ satisfies Condition (K) if $S/I$ satisfies Condition (L) for every saturated ideal $I$ of $S$. Fittingly, the inverse semigroup $\mcS_{\Lambda}$ of a directed 
graph $\Lambda$ satisfies Condition (K) if and only if $\Lambda$ satisfies Condition (K). Under a couple of additional assumptions on the semilattice of an inverse semigroup, 
we show that the tight groupoid of $S$ is strongly effective if and only if $S$ satisfies Condition (K).

After presenting the necessary background information in Section 2, we begin our investigation in Section 3 by defining the congruence $\leftrightarrow$ on $S$ and showing 
that $\Gt(S) \cong \Gt(S/\!\!\lr)$ when $S$ is Hausdorff. In Theorem \ref{thm:condL} we state some consequences of Condition (L) for an inverse semigroup.

In Section 4, we develop the connection between ideals in $S$ and closed invariant sets of the character spaces, $\fE$ and $\tfE$, which are the unit spaces of the universal 
groupoid $\G(S)$ and the tight groupoid $\Gt(S)$, respectively. We demonstrate a correspondence between invariant ideals in $E(S)$ and certain closed invariant subsets of 
$\fE$, which sends an ideal $X$ of $E(S)$ to the set $h(X)$ of all filters disjoint from $X$. Each ideal $I$ of $S$ induces an invariant ideal $X = I \cap E$ of $E(S)$, and we 
show in Theorem \ref{thm:universalreduction} that $\G(S/I)$ is isomorphic to the reduction $\G(S) \vert_{h(X)}$, and similarly $\G(I)$ is isomorphic to $\G(S) \vert_{h(X)^{c}}$. 
It is rarely the case that all closed invariant subsets of $\fE$ come from an ideal, but the situation is better in the case of the tight groupoid. There we obtain a bijective 
correspondence between saturated ideals in $S$ and closed invariant subsets of $\tfE$, provided we make two finiteness assumptions on the semilattice $E(S)$ (see 
Theorem \ref{thm:tightideals}).

Section 5 is devoted to the study of Condition (K) for inverse semigroups. In Theorem \ref{thm:condK}, we show that if $S$ is a Hausdorff inverse semigroup satisfying the 
aforementioned finiteness conditions, then $S$ satisfies Condition (K) if and only if $\Gt(S)$ is strongly effective. We can then describe the ideals in the Steinberg algebra 
and $C^*$-algebra of $\Gt(S)$ and state conditions for simplicity. Our definition of Condition (K) seems difficult to verify, since it involves every possible quotient of $S$ by a 
saturated ideal. With this in mind, we develop a sufficient condition---every congruence on $S$ is a Rees congruence. We characterize this condition in Theorem \ref{thm:Rees}.

In Section 6, we study this sufficient condition in greater detail for the inverse semigroup $\mcS_{G,\Lambda}$ of a self-similar graph action $(G,\Lambda,\varphi)$. We 
characterize the ideals of $\mcS_{G,\Lambda}$ in terms of hereditary, $G$-invariant sets of vertices. We prove that every quotient of $\mcS_{G,\Lambda}$ is itself the inverse 
semigroup of a self-similar graph action, and we characterize the property that all congruences on $\mcS_{G,\Lambda}$ are Rees.

Finally, we return to the uniqueness theorems of \cite{LaLondeMilan} in Section 7. It turns out that some of the results in this paper allow us to clarify the hypotheses used in those theorems. 

\section{Preliminaries}

This paper is concerned largely with the ideal structure of inverse semigroups. Recall that an \textit{inverse semigroup} is a semigroup $S$ such that for each $s$ in $S$, there 
exists a unique $s^*$ in $S$ such that 
\[
	s = s s^* s\quad \text{and}\quad s^* = s^* s s^*.
\]
We assume all inverse semigroups in this paper are countable. The elements $e \in S$ satisfying $e^2 = e$ (and hence $e^* = e$) are called \emph{idempotents}. The set of all 
idempotents in $S$ is denoted by $E(S)$.

There is a natural partial order on $S$ defined by $s \leq t$ if $s = te$ for some idempotent $e$. Given $s$ in $S$, it will often be useful to consider the sets
\[
	s^\down = \{ t \in S : t \leq s\}, \quad s^\up = \{ t \in S : t \geq s\}.
\] 
It is also worth noting that the subsemigroup $E(S)$ of idempotents is commutative, and hence forms a meet semilattice under the natural partial order with 
$e \wedge f= ef$ for $e,f$ in $E(S)$.

There are a number of useful equivalence relations defined on an inverse semigroup. The two that play a particularly important role here are the $\mcH$- and 
$\mu$-relations. The $\mcH$-relation is defined by $s \mcH t$ if and only if $s^* s = t^* t$ and $s s^* = t t^*$. The $\mcH$-class of an idempotent $e$, denoted $H_e$, 
is the maximum subgroup of $S$ with identity $e$. The $\mu$-relation is defined by $s \mcmu t$ if and only if $ses^* = tet^*$ for all $e \in E$. We denote the 
$\mu$-class of an idempotent $e$ by $Z_e$; it is also a group with identity $e$, hence a subgroup of $H_e$. In fact, it is well-known that $\mu \subseteq \mathcal{H}$. 
If $\mu = \mathcal{H}$, we say that $S$ is \emph{cryptic}.

An equivalence relation $\rho$ on $S$ is a {\it congruence} if $(a,b), (c,d) \in \rho$ implies $(ac, bd) \in \rho$. The quotient of an inverse semigroup by a congruence is 
again an inverse semigroup. The $\mu$-relation is always a congruence, though $\mcH$ need not be in general. A relation $\rho$ is \emph{idempotent separating} if 
each $\rho$-class contains at most one idempotent. Both $\mcH$ and $\mu$ are idempotent separating. Moreover, it is well known that $\mu$ is the maximal idempotent 
separating congruence. Notice then that $S$ is cryptic if and only if $\mcH$ is a congruence. 

A nonempty subset $I$ of $S$ is called a \emph{left ideal} if $SI \subseteq I$, a \emph{right ideal} if $IS \subseteq I$, and an \emph{ideal} if $SIS \subseteq I$. An inverse 
semigroup is \emph{simple} if it has no proper ideals. We say $S$ is \emph{$0$-simple} if it contains a 0, a nonzero element, and no proper nonzero ideals.

Given an ideal $I$ of $S$, there is a congruence $\rho_I$ on $S$ defined by $a \mathrel{\rho_I} b$ if $a,b \in I$ or $a = b$. We call $\rho_I$ the \emph{Rees congruence} of $I$. The quotient of $S$ by $\rho_I$ is denoted $S/I$ 
and is called the \emph{Rees quotient} of $S$ by $I$. Note that $S/I$ is simply the set $S \setminus I \cup \{0\}$ endowed with the product
\[
	s \cdot t = \begin{cases}
		st & \text{ if } st \in S \backslash I \\
		0 & \text{ otherwise}
	\end{cases}
\]
and inversion inherited from that of $S$.

\subsection{Groupoids} Recall that a \emph{groupoid} consists of a set $G$ together with a distinguished set $G^{(2)} \subseteq G \times G$ of composable pairs, and an associative map $(x, y) \mapsto xy$ from $G^{(2)} \to G$. There is also an involution $x \mapsto x^{-1}$ from $G \to G$ satisfying
\[
	x^{-1}(xy) = y, \quad (xy)y^{-1} = x
\]
whenever $(x, y) \in G^{(2)}$. The set
\[
	G^{(0)} = \bigl\{ u \in G : u^2 = u = u^{-1} \bigr\}
\]
is the \emph{unit space} of $G$, and its elements are called \emph{units}. The surjections $r, d : G \to G^{(0)}$ defined by
\[
	r(x) = xx^{-1}, \quad d(x) = x^{-1} x
\]
are called the \emph{range} and \emph{source} maps, respectively. We will assume that $G$ is equipped with a locally compact (but not necessarily Hausdorff)
topology making multiplication and inversion continuous. Note that $r$ and $d$ are then automatically continuous. Indeed, for all groupoids in this paper it will be
the case that $r, d : G \to G^{(0)}$ are local homeomorphisms, i.e., that $G$ is an \emph{\'{e}tale} groupoid.

Given $u \in G^{(0)}$, the set of elements $x \in G$ satisfying $r(x) = d(x) = u$ is called the \emph{isotropy group} at $u$, denoted by $G_u^u$. The \emph{isotropy bundle}
of $G$ is the set 
\[
	\iso(G) = \bigcup_{u \in G^{(0)}} G_u^u.
\]
If $\iso(G) = G^{(0)}$, then $G$ is said to be \emph{principal}. We say that an \'{e}tale groupoid $G$ is \emph{effective} if $\iso(G)^\circ = G^{(0)}$. If $G$ is Hausdorff and
second countable, this condition is equivalent to the stipulation that $G$ is \emph{topologically principal}---the set of units satisfying $G_u^u = \{u\}$ is dense in $G^{(0)}$.

A set $A \subseteq G^{(0)}$ is \emph{invariant} if $d(x) \in A$ implies $r(x) \in A$ for all $x \in G$. If $A$ is an invariant set of units, the set
\[
	G \vert_A = r^{-1}(A) = d^{-1}(A)
\]
is a subgroupoid of $G$, called the \emph{reduction} to $A$. An \'{e}tale groupoid is said to be \emph{strongly effective} if $G \vert_A$ is effective for every closed invariant 
set $A \subseteq G^{(0)}$.

\begin{exmp}
	The only groupoids we will consider in this paper are the ones built from an inverse semigroup $S$. Recall that a \emph{filter} in $E = E(S)$ is a nonempty set 
	$F \subseteq E$ that is closed under multiplication and closed upward in the partial order on $E$. We let $\fE$ denote the set of all filters in $E$ not containing 0.
	Given $e, e_1, \ldots, e_n \in E$, we define
	\begin{equation}
	\label{eqn:nes}
		N(e; e_1, \ldots, e_n) = \bigl\{ F \in \fE : e \in F, \, e_1, \ldots, e_n \not\in F \bigr\}.
	\end{equation}
	For each $e \in E$, it will be helpful if we let
	\[
		N^e = \bigl\{ F \in \fE : e \in F \bigr\}, \quad N_e = \bigl\{ F \in \fE : e \not\in F \bigr\}.
	\]
	The sets defined in \eqref{eqn:nes} form a base of compact open sets for a locally compact, Hausdorff topology on $\fE$.
	
	There is a natural action of $S$ on its filter space $\fE$, which is defined in terms of the maps $\beta_s : N^{s^*s} \to N^{ss^*}$ given by
	\[
		\beta_s(F) = (sFs^*)^\up.
	\]
	Now we define an equivalence relation on the set
	\[
		S * \fE = \bigl\{ (s, F) \in S \times \fE : s^*s \in F \bigr\}
	\]
	by declaring $(s, F) \sim (t, F')$ if and only if $F = F'$ and there exists $e \in F$ such that $se = te$. We denote the equivalence class of $(s, F)$ by $[s, F]$, and 
	we define $\G(S) = (S*\fE)/\sim$. This set is a groupoid with respect to the operations
	\[
		[t, \beta_s(F)][s, F] = [ts, F], \quad [s, F]^{-1} = [s^*, \beta_s(F)].
	\]
	The range and source maps are given by
	\[
		r([s, F]) = [ss^*, \beta_s(F)], \quad d([s, F]) = [s^*s, F],
	\]
	so there is a natural identification of $\G(S)^{(0)}$ with $\fE$. Sets of the form
	\[
		\Theta(s, U) = \bigl\{ [s, F] \in \G(S) : F \in U \bigr\},
	\]
	where $s \in S$ and $U \subseteq \fE$ is open, form a base for a topology that makes $\G(S)$ into an \'{e}tale (but not necessarily Hausdorff) groupoid. We call 
	$\G(S)$ the \emph{universal groupoid} of $S$.
	
	There is another groupoid associated to $S$ which was introduced by Exel in \cite{ExelTight} and aligns more closely with the structure of $S$ in certain respects. We let $\tfE$ denote the closure of the set of ultrafilters in $\fE$, which is called the \emph{tight filter space} of $S$. For notational clarity, we denote these relative basic open sets in $\tfE$ by
	\[ 
		D(e; e_1,e_2, \dots, e_n) = N(e; e_1,e_2, \dots, e_n) \cap \tfE
	\]
	for $e, e_1, \ldots, e_n \in E$. Similarly, for each $e \in E$ we define
	\[
		D^{e} = N^{e} \cap \tfE, \quad D_{e} = N_{e} \cap \tfE
	\]
	Note that $\tfE$ is a closed, invariant subset of $\fE = \G(S)^{(0)}$. The reduction of $\G(S)$ to the tight filter space is called the \emph{tight groupoid} of $S$, 
	denoted by $\Gt(S)$.
\end{exmp}

\subsection{Directed graphs}

The primary motivating example of an inverse semigroup in this paper is the inverse semigroup $\mcS_{\Lambda}$ of a directed graph $\Lambda$. Briefly, a directed graph $\Lambda = (\Lambda^0, \Lambda^1, r, s)$ consists of countable sets $\Lambda^0$, $\Lambda^1$ and functions $r,s : \Lambda^1 \to \Lambda^0$ called the \emph{range} and \emph{source} maps, respectively. The elements of $\Lambda^0$ are called \emph{vertices}, and the elements of $\Lambda^1$ are called \emph{edges}. Given an edge $e$, $r_{e}$ denotes the range vertex of $e$ and $s_{e}$ denotes the source vertex. We denote by $\Lambda^*$ the collection of finite directed paths in $\Lambda$. The range and source maps $r,s$ can be extended to $\Lambda^*$ by defining $r_{\alpha} = r_{\alpha_n}$ and $s_{\alpha} = s_{\alpha_1}$ for a path $\alpha = \alpha_n \alpha_{n-1} \cdots \alpha_1$ in $\Lambda^*$. If $\alpha = \alpha_n \alpha_{n-1} \cdots \alpha_1$ and $\beta = \beta_m \beta_{m-1} \cdots \beta_1$ are paths with $s_{\alpha} = r_{\beta}$, we write $\alpha \beta$ for the path $\alpha_n \cdots \alpha_1 \beta_m \cdots \beta_1$. 

The \textit{graph inverse semigroup} of the directed graph $\Lambda$ is the set
\[ \mcS_{\Lambda} = \{(\alpha,\beta) : s_{\alpha} = s_{\beta} \} \cup \{ 0 \} \]
with the products not involving zero defined by
\[ (\alpha,\beta) (\gamma,\nu) = \left\{\begin{array}{ll}
        (\alpha \gamma', \nu) & \mbox{if $\gamma = \beta \gamma'$} \\
        (\alpha,\nu\beta') & \mbox{if $\beta = \gamma \beta'$} \\
        0 & \mbox{otherwise}
   \end{array} \right. \]
The inverse is given by $(\alpha,\beta)^{*} = (\beta,\alpha)$. It is worth noting that if $\mcS_{\Lambda}$ is the inverse semigroup of a directed graph $\Lambda$, then the tight 
groupoid of $\mcS_{\Lambda}$ is isomorphic to the usual path groupoid of $\Lambda$.

\section{Condition (L) for Inverse Semigroups}

A crucial result in the study of graph $C^*$-algebras is the Cuntz-Krieger uniqueness theorem, which says that, under certain conditions, a homomorphism out of a graph 
$C^*$-algebra is injective if and only if it does not vanish on any vertex projection. The key hypothesis is that every cycle in the graph has an entrance, which is known in
the literature as Condition (L). 

It turns out that a graph satisfies Condition (L) if and only if its path groupoid is effective \cite[Lemma 3.4]{KPR}. With this result in mind, the appropriate analogue of 
Condition (L) for an inverse semigroup $S$ ought to be an algebraic condition that is equivalent to $\Gt(S)$ being effective. There is already a result in this direction, but it 
requires an additional assumption on the inverse semigroup.

\begin{defn}
	Let $E$ be a meet semilattice. We say that $E$ is \emph{0-disjunctive} if given $e, f \in E$ with $0 < e < f$, there exists $0 < e' < f$ such that $ee' = 0$.
\end{defn}

In \cite[Corollary 5.5]{LaLondeMilan}, it is shown that if $S$ is a fundamental inverse semigroup with a $0$-disjunctive semilattice of idempotents, then the interior of  
$\iso (\Gt(S))$ is just $\Gt(S)^{(0)}$. In other words, $\Gt(S)$ is effective. That result seems to be of little relevance in determining the appropriate analog of Condition (L) for 
all inverse semigroups, since many important examples have a semilattice that fails to be $0$-disjunctive. For example, the inverse semigroup $\mcS_\Lambda$ of a directed 
graph $\Lambda$ has a $0$-disjunctive semilattice if and only if no vertex has in-degree 1. 

However, in this section we define a congruence $\leftrightarrow$ on an inverse semigroup $S$ that is directly related to the $0$-disjunctive condition on $E(S)$. We show 
that $S / \!\!\lr$ has a $0$-disjunctive semilattice and, assuming $S$ is Hausdorff, that $\Gt(S) \cong \Gt(S / \!\!\lr)$. Therefore, for Hausdorff $S$, $\Gt(S)$ is effective if and only if 
$S / \!\!\lr$ is fundamental. Thus we believe it is appropriate to say that $S$ satisfies Condition (L) if and only if $S / \!\!\lr$ is fundamental.

The congruence $\leftrightarrow$ that we are about to define was introduced by Lenz in \cite{lenz}, and has also been studied in \cite{LawsonGraph} and \cite{lawson-lenz}. In \cite{lawson-lenz} the notion of \emph{coverages} is used to produce the desired congruence, while \cite{LawsonGraph} requires one to work with inverse semigroups in which any two elements have 
a meet. In addition, the groupoids in \cite{lawson-lenz} are defined differently from the ones used by $C^*$-algebraists. Therefore, we will develop $\leftrightarrow$ in full detail.

\begin{defn}
	Let $S$ be an inverse semigroup with 0. Given $a, b \in S$, we define $a \rightarrow b$ if whenever $0 < x \leq a$, ${x^\downarrow} \cap {b^\downarrow} \neq 0$. We write 
	$a \leftrightarrow b$ if $a \rightarrow b$ and $b \rightarrow a$.
\end{defn}

Note that our definition of $a \rightarrow b$ says precisely that $a \rightarrow \{b\}$, in the sense of the relation described in \cite[Example 4.1(4)]{lawson-lenz}.

\begin{rem}
	It is worthwhile to discuss how our relation $\leftrightarrow$ is connected to the one induced by the tight coverage $\mathcal{T}$ in \cite{lawson-lenz}. Let $\equiv_\mathcal{T}$ 
	denote the Lawson-Lenz relation, which is defined by $a \equiv_{\mathcal{T}} b$ if $\mathcal{T}(a) \cap \mathcal{T}(b) \neq \emptyset$. Here
	\[
		\mathcal{T}(a) = \{ A \subseteq {a^\downarrow} : A \text{ is finite and } a \rightarrow A\},
	\] 
	where $a \rightarrow A$ if given any $0 < x \leq a$, there exists $b \in A$ with ${x^\downarrow} \cap {b^\downarrow} \neq 0$. 
	
	We claim that the relation $\equiv_\mathcal{T}$ is always contained in $\leftrightarrow$. Suppose $a \equiv_\mathcal{T} b$. Then there exists a nonempty, finite set 
	$X \subseteq \mathcal{T}(a) \cap \mathcal{T}(b)$. In particular, this means that $X \subseteq {a^\downarrow}$ and $a \rightarrow X$, and similarly 
	$X \subseteq {b^\downarrow}$ and $b \rightarrow X$. Therefore, given $0 < x \leq a$, there exists $y \in X$ such that ${x^\downarrow} \cap {y^\downarrow} \neq 0$. 
	But ${y^\downarrow} \subseteq {b^\downarrow}$, so it follows that ${x^\downarrow} \cap {b^\downarrow} \neq 0$. Thus $a \rightarrow b$. A similar argument shows 
	that $b \rightarrow a$, so $a \leftrightarrow b$.
	
	On the other hand, it is not clear whether $a \leftrightarrow b$ necessarily implies $a \equiv_{\mathcal{T}} b$. However, we can prove it when $S$ is Hausdorff. 
	Suppose $a \leftrightarrow b$. Then ${a^\downarrow} \cap {b^\downarrow}$ is finitely generated, say by $Z = \{z_1, z_2, \ldots, z_n\}$. Given $0 < x \leq a$, we have 
	${x^\downarrow} \cap {b^\downarrow} \neq 0$. That is, there exists $y \neq 0$ such that
	\[
		y \in {x^\downarrow} \cap {b^\downarrow} \subseteq {a^\downarrow} \cap {b^\downarrow},
	\]
	so $y \leq z_i$ for some $i$. Thus ${x^\downarrow} \cap {z_i^\downarrow} \neq 0$, which shows that $a \rightarrow Z$. Similarly, $b \rightarrow Z$. Therefore, 
	$Z \in \mathcal{T}(a) \cap \mathcal{T}(b)$, so $a \equiv_{\mathcal{T}} b$.
\end{rem}

The proof of the next proposition is (unsurprisingly) very similar to the verification that $\mathcal{D}$ satisfies the axioms of a coverage in \cite[Example 4.1(4)]{lawson-lenz}.

\begin{prop}
	The relation $\lr$ defines a 0-restricted congruence on $S$.
\end{prop}
\begin{proof}
	It is clear that $a \lr a$ for all $a \in S$, since ${x^\downarrow} \cap {a^\downarrow} = {x^\downarrow} \neq 0$ whenever $0 < x \leq a$. It is also immediate that 
	$\lr$ is symmetric. Now suppose $a \lr b$ and $b \lr c$, and let $0 < x \leq a$. Then we have ${x^\downarrow} \cap {b^\downarrow} \neq 0$, so there exists $y \neq 0$ 
	such that $y \in {x^\downarrow} \cap {b^\downarrow}$. Since $b \lr c$, we must have ${y^\downarrow} \cap {c^\downarrow} \neq 0$. Thus
	\[
		0 \neq {y^\downarrow} \cap {c^\downarrow} \subseteq {x^\downarrow} \cap {c^\downarrow},
	\]
	and it follows that $a \lr c$. Thus $\lr$ defines an equivalence relation on $S$.
	
	Now we show that $\lr$ is a congruence. Suppose $a \lr b$ and $c \lr d$, and let $0 < x \leq ac$. Then it is easily checked that $aa^*x = x$, so $a^*x \neq 0$ and
	\[
		0 \leq a^*x \leq a^*ac \leq c.
	\]
	Since $c \lr d$, we have $(a^*x)^\down \cap d^\down \neq 0$. That is, there exists $y \neq 0$ such that $y \leq a^*x$ and $y \leq d$. It is again easy to verify
	that $a^*a y = y$, so $ay \neq 0$. Thus 
	\[
		0 < ay \leq ad, \quad 0 < ay \leq aa^*x = x.
	\]
	From the first inequality it follows that $ayd^*d = ay$ and
	\[
		0 < ayd^* \leq add^* \leq a,
	\]
	so $(ayd^*)^\down \cap b^\down \neq 0$. If we let $z \neq 0$ such that $z \leq ayd^*$ and $z \leq b$, it is straightforward to check that $zdd^* = z$, so $zd \neq 0$.
	Thus
	\[
		0 < zd \leq bd, \quad 0 < zd \leq ayd^*d = ay \leq x,
	\]
	so $x^\down \cap (bd)^\down \neq 0$. Therefore, $ac \lr bd$, and $\lr$ is a congruence.
	
	Finally, notice that if $a \rightarrow 0$, then it must be the case that $a = 0$. Thus $a \lr 0$ if and only if $a=0$, so $\lr$ is a 0-restricted congruence.
\end{proof}

Since $\lr$ is a congruence, we may form the quotient $S/\!\!\lr$. We claim that by doing so, we obtain an inverse semigroup with a 0-disjunctive semilattice. 

\begin{prop}
\label{prop:0disjunctive}
	Let $S$ be an inverse semigroup with 0. 
\begin{enumerate}
	\item The semilattice of $S / \! \! \leftrightarrow$ is $0$-disjunctive.
	
	\item If $\lr$ is equality, then $E(S)$ is 0-disjunctive.
	
	\item If $S$ is fundamental and $E(S)$ is 0-disjunctive, then $\lr$ is equality.
\end{enumerate}
\end{prop} 
\begin{proof}
	Let $[s]$ denote the equivalence class of $s \in S$ in $S / \!\!\lr$. For (1), suppose that $0 \neq [e] < [f]$ where $e,f$ are idempotents in $S$. Then $ef \leftrightarrow e$ 
	and it follows that $e \rightarrow f$. Since $[e] \neq [f]$, $f \not\rightarrow e$. So there exists $0 < f' < f$ such that $f'e = 0$. Note that $f \not\rightarrow f'$ since 
	$ef' = 0$. Thus $0 < [f'] < [f]$ and $[f'][e] = 0$. Thus $E(S / \!\!\lr)$ is 0-disjunctive. Clearly (2) now follows as a special case, since $S/ \! \! \lr$ is equal to $S$ when 
	$\lr$ is equality.
	
	For (3), suppose $S$ is fundamental and $E(S)$ is 0-disjunctive. Let $e$ and $f$ be nonzero idempotents with $e \lr f$. If $e \neq f$ then either $ef < f$ or $ef < e$. 
	Suppose without loss of generality that $ef < f$, and notice that $ef \neq 0$ since $\leftrightarrow$ is $0$-restricted. Since $E(S)$ is $0$-disjunctive, there exists 
	$0 \neq e' < f$ such that $efe' = 0$. Then $e \lr f$ implies $0 = efe' \lr e'$, which is a contradiction. Thus $e = f$. Now suppose $s \lr t$ for some nonzero $s, t \in S$. Then 
	$s^* e s \lr t^* e t$ for all $e \in E(S)$, so $s \mcmu t$. Since $S$ is fundamental, $s = t$.
\end{proof}

Our goal now is to analyze the tight groupoid of $S/\! \! \lr$. We first claim that the tight filter spaces of $S$ and $S/\! \! \lr$ are homeomorphic. Toward this end, we 
let $\tau : S \to S/\! \! \lr$ denote the quotient homomorphism, and we define a map $\theta : {E(S)_0\sphat} \to E(S/\!\!\lr)_0\sphat$ by $\theta(F) = \tau(F)^{\uparrow}$. 

\begin{lem}
	If $F \in E(S)_0\sphat$ is an ultrafilter, then $\theta(F)$ is an ultrafilter.
\end{lem}
\begin{proof}
	For an ultrafilter $F \in E(S)_0\sphat$, note that $\theta(F)$ is a filter. Next, suppose $f \in E(S/\! \!\lr)$ satisfies $fe \neq 0$ for all $e \in \theta(F)$. In particular, for all $e' \in F$ we have $f \tau(e') \neq 0$. 
	Choose $f' \in E(S)$ such that $f = \tau(f')$. Then for all $e' \in F$ we have
	\[
		\tau(f' e') = \tau(f') \tau(e') = fe \neq 0,
	\]
	which implies that $f'e' \neq 0$. Since $F$ is an ultrafilter, $f' \in F$ by \cite[Lemma 3.3]{exel}. It follows that $f = \tau(f') \in \theta(F)$. Thus $\theta(F)$ is an ultrafilter.
\end{proof}

\begin{lem}
\label{lem:uf}
	Let $F \in E(S)_0\sphat$ be an ultrafilter. An idempotent $e \in E(S)$ belongs to $F$ if and only if $\tau(e) \in \theta(F)$. Consequently, $\theta(F) = \tau(F)$ whenever 
	$F$ is an ultrafilter. 
\end{lem}
\begin{proof}
	Obviously $e \in F$ implies $\tau(e) \in \theta(F)$. On the other hand, suppose $\tau(e) \in \theta(F)$. Then for some $f \in F$, $\tau(f) \leq \tau(e)$. So $ef \lr f$, and for all $e' \in F$ we have
	\[
		\tau(efe') = \tau(f) \tau(e') \neq 0.
	\]
	Thus $ee' \neq 0$, so \cite[Lemma 3.3]{exel} implies $e \in F$ since $F$ is an ultrafilter.
\end{proof}

For brevity, we let $E(S)_u\sphat$ and $E(S/\! \!\lr)_u\sphat$ denote the sets of ultrafilters in $E(S)_0\sphat$ and $E(S/\! \!\lr)_0\sphat$, respectively. 

\begin{lem}
\label{lem:uf-homeo}
	The map $\theta$ restricts to a homeomorphism between the ultrafilter spaces $E(S)_u\sphat$ and $E(S/ \! \! \lr)_u\sphat$.
\end{lem}
\begin{proof}
	Suppose $F, F' \in E(S)_u\sphat$ such that $\theta(F) = \theta(F')$. Let $e \in F$. Then $\tau(e) \in \theta(F) = \theta(F')$, so $e \in F'$ by the previous lemma. 
	Therefore, $F \subseteq F'$. By symmetry, $F' \subseteq F$. Thus $F = F'$, and $\theta$ is injective on $E(S)_u\sphat$.
	
	Now suppose $F' \in E(S/\! \! \lr)_u\sphat$ and put $F = \tau^{-1}(F')$. We claim that $F$ is an ultrafilter. One can quickly check that $F$ is a filter. Next let $e \in E(S)$ such that $ef \neq 0$ for all $f \in F$. Then $\tau(e) \tau(f) = \tau(ef) \neq 0$ for all $f \in F$. Thus $\tau(e) \in F'$ 
	since $F'$ is an ultrafilter, so $e \in F$. Therefore, $F \in E(S)_u\sphat$. It is now immediate that $\theta(F) = \tau(F) = F'$, so $\theta$ maps $E(S)_u\sphat$ onto 
	$E(T)_u\sphat$.
	
	It remains to see that the restriction of $\theta$ to $E(S)_u\sphat$ is continuous and open. If $e, e_1, \ldots, e_n \in E(S)$, then we clearly have
	\[
		\theta \bigl( N(e, {e_1, \ldots, e_n}) \cap E(S)_u\sphat \bigr) = N(\tau(e), \tau(e_1), \ldots, \tau(e_n)) \cap E(S/\! \! \lr)_u\sphat
	\]
	by Lemma \ref{lem:uf}. Thus $\theta$ takes basic open sets to basic open sets, and it follows that $\theta$ defines a homeomorphism of $E(S)_u\sphat$ onto 
	$E(T)_u\sphat$.
\end{proof}

Now we extend the preceding results for ultrafilters to tight filters. Our first order of business is to check that $\theta$ maps $E(S)_{{\rm{tight}}}\sphat$ into 
$E(S/\!\!\lr)_{{\rm{tight}}}\sphat$.

\begin{lem}
\label{lem:tfextension}
	Let $F \in E(S)_{{\rm{tight}}}\sphat$, and suppose $\{F_i\}_{i=1}^\infty$ is a sequence of ultrafilters converging to $F$. Then $\theta(F_i)$ converges to $\theta(F)$
	in $E(S/\!\!\lr)_{{\rm{tight}}}\sphat$. Hence $\theta(F)$ is a tight filter and $\theta$ restricts to a map 
	$\theta : E(S)_{{\rm{tight}}}\sphat \to E(S/\!\!\lr)_{{\rm{tight}}}\sphat$.
\end{lem}
\begin{proof}
	Suppose $e \in E(S)$ with $\tau(e) \in \theta(F)$. Then there exists $f \in F$ such that $\tau(f) \leq \tau(e)$. In other words, $F \in N^f$. Since $F_i \to F$, the
	$F_i$ eventually belong to $N^f$. Hence $\theta(F_i) \in N^{\tau(f)} \subseteq N^{\tau(e)}$ eventually.
	
	Now suppose $\tau(e) \not\in \theta(F)$. Then $e \not\in F$, so $F \in N_e$. It follows that $F_i \in N_e$ eventually, so $\theta(F_i) \in N_{\tau(e)}$ eventually by
	Lemma \ref{lem:uf}.
	
	Based on these two observations, it is now evident that $\theta(F_i)$ eventually belongs to any basic open set containing $\theta(F)$. Hence $\theta(F_i) \to \theta(F)$, 
	so $\theta(F) \in E(S/\!\!\lr)_{{\rm{tight}}}\sphat$ as it is a limit of ultrafilters.
\end{proof}

We claim that $\theta$ yields a homeomorphism of $E(S)_{\text{tight}}\sphat$ onto $E(S/\!\!\lr)_{\text{tight}}\sphat$. To prove it, we need an extension of Lemma \ref{lem:uf} to 
tight filters.

\begin{lem}
\label{lem:tf}
	Let $F \in E(S)_{{\rm{tight}}}\sphat$. For each $e \in E(S)$, we have $e \in F$ if and only if $\tau(e) \in \theta(F)$. Consequently, $\theta(F) = \tau(F)$ whenever $F$ is a tight
	filter. 
\end{lem}
\begin{proof}
	Clearly $e \in F$ implies $\tau(e) \in \theta(F)$. Conversely, suppose $e \in E(S)$ with $\tau(e) \in \theta(F)$. Let $\{F_i\}_{i=1}^{\infty}$ be a sequence of ultrafilters
	converging to $F$ in $E(S)_{\text{tight}}\sphat$. Lemma \ref{lem:tfextension} implies that $\theta(F_i) \to \theta(F)$, so $\theta(F_i)$ eventually belongs to the 
	basic open set $N^{\tau(e)}$. Thus $\tau(e) \in \theta(F_i)$ eventually. Since $\theta(F_i)$ is an ultrafilter, $e \in F_i$ eventually by Lemma \ref{lem:uf}, i.e., the $F_i$ eventually all belong
	to the basic open set $N^e$. But $N^e$ is closed, so it follows that $e \in F$.
\end{proof}

There is a characterization of tight filters, observed by Lawson \cite{LawsonStone} and Exel \cite[Proposition 11.8]{ExelTight}, that will now prove useful: A filter $F$ is 
tight if and only if $e \covby  \{e_1, e_2, \dots, e_n\}$ and $e \in F$ implies $e_i \in F$ for some $i$.

\begin{prop}
\label{prop:tighthomeo}
	The map $\theta$ restricts to a homeomorphism of $E(S)_{{\rm{tight}}}\sphat$ onto $E(S/\!\!\lr)_{{\rm{tight}}}\sphat$.
\end{prop}
\begin{proof}
	Thanks to Lemma \ref{lem:tf}, the proof that $\theta$ is injective on $E(S)_{\text{tight}}\sphat$ is identical to the one for ultrafilters from Lemma \ref{lem:uf-homeo}. 
	To see that $\theta$ maps $E(S)_{{\rm{tight}}}\sphat$ onto $E(S/\!\!\lr)_{{\rm{tight}}}\sphat$, suppose $F' \in E(S/\!\!\lr)_{\text{tight}}\sphat$ and let $F = \tau^{-1}(F')$. 
	Let $e \in F$, and suppose $e \covby \{e_1, \ldots, e_n\}$. Then $\tau(e) \in F'$, and we claim that $\tau(e) \covby \{\tau(e_1), \ldots, \tau(e_n)\}$. Clearly 
	$\tau(e_i) \leq \tau(e)$ for all $i$. Suppose $\tau(f) \leq \tau(e)$ for some $f \in E(S)$. Then $\tau(f) = \tau(ee')$ for some $e' \in E(S)$, so $f \lr ee'$. Thus there exists 
	a nonzero idempotent $g \in {f^\down} \cap {(ee')^\down}$. In particular, $g \leq e$, so $ge_i \neq 0$ for some $i$. Therefore, $fe_i \neq 0$ for some $i$, whence 
	$\tau(f) \tau(e_i) \neq 0$ for some $i$. Thus $\tau(e) \covby \{\tau(e_1), \ldots, \tau(e_n)\}$. Since $F'$ is a tight filter, $\tau(e_i) \in F'$ for some $i$. It follows from 
	Lemma \ref{lem:tf} that $e_i \in F$, so $F$ is a tight filter. Thus $\theta : E(S)_{\text{tight}}\sphat \to E(S/\!\!\lr)_{\text{tight}}\sphat$ is surjective.
	
	As in Lemma \ref{lem:uf-homeo}, it is easy to see that $\theta$ takes basic open sets in $E(S)_{\text{tight}}\sphat$ to basic open sets in $E(S/\!\!\lr)_{\text{tight}}\sphat$. 
	Thus $\theta$ implements a homeomorphism between the tight filter spaces.
\end{proof}

The preceding lemmas do not require us to assume $S$ is Hausdorff. However, this hypothesis is necessary for the proof that the groupoids $\Gt(S)$ and $\Gt(S/\!\!\lr)$
are isomorphic. Lenz has proved a very similar result, \cite[Remark 6.10]{lenz}, under the assumption that meets exist in $S$. 

\begin{thm}\label{thm:groupoiddoublearrow}
	Let $S$ be a Hausdorff inverse semigroup, and let $\tau : S \to S/\!\!\lr$ be the quotient map. The map $\Phi : \Gt(S) \to \Gt(S/\!\!\lr)$ defined by
	\[
		\Phi([s, F]) = [\tau(s), \theta(F)]
	\]
	is an isomorphism of topological groupoids.
\end{thm}
\begin{proof}
	It is straightforward to check that $\Phi$ is well-defined. To show that $\Phi$ is a homomorphism, we need to know that $\theta(\beta_s(F)) = \beta_{\tau(s)}(\theta(F))$
	for all $F \in E(S)_{{\rm{tight}}}\sphat$ and $s \in S$ with $s^*s \in F$. If $e \in \beta_s(F)$, then $e \geq sfs^*$ for some $f \in F$, so $\tau(e) \geq \tau(s) \tau(f) \tau(s)^*$.
	Hence $\tau(e) \in \beta_{\tau(s)}(\theta(F))$, so $\theta(\beta_s(F)) \subseteq \beta_{\tau(s)}(\theta(F))$. The reverse containment is similar: if 
	$\tau(e) \in \beta_{\tau(s)}(\theta(F))$, then $\tau(e) \geq \tau(s) \tau(f) \tau(s)^* = \tau(sfs^*)$ for some $f \in F$, whence $\tau(e) \in \theta(\beta_s(F))$.
	It is now easy to see that
	\begin{align*}
		\Phi([t, \beta_s(F)]) \Phi([s, F]) &= [\tau(t), \beta_{\tau(s)}(\theta(F))][\tau(s), \theta(F)] \\
			&= [\tau(ts), \theta(F)] \\
			&= \Phi([ts, F]) \\
			&=\Phi([t, \beta_s(F)][s, F]),
	\end{align*}
	so $\Phi$ is a homomorphism.
	
	Note that $\Phi$ is surjective since $\theta$ and $\tau$ are both surjective. Furthermore, $\Phi$ is injective on the unit space of $\Gt(S)$ by Proposition 
	\ref{prop:tighthomeo}. Suppose $[s, F] \in \ker \Phi$, i.e., $[\tau(s), \tau(F)] = [e', F']$ for some $F' \in E(S/\!\!\lr)_{\text{tight}}\sphat$ and $e' \in F'$. Then 
	$\tau(F) = F'$, and we have $e' = \tau(e)$ for some $e \in F$. Moreover, there exists $f \in F$ such that $\tau(sf) = \tau(ef)$, meaning that $sf \lr ef$. 
	In particular, ${(sf)^\down} \cap {(ef)^\down} \neq 0$, and in fact
	\[
		sf^\downarrow \cap ef^\downarrow = \{e_1, e_2, \ldots, e_n\}^\downarrow
	\]
	for some collection of nonzero idempotents $e_1, \ldots, e_n \in E(S)$ since $S$ is Hausdorff. If $0 < x \leq ef$, then $x^\downarrow \cap sf^\downarrow \neq 0$, 
	and we can find a nonzero idempotent $y \in x^\downarrow$ such that $y \leq e_i$ for some $i$. Hence $xe_i \geq y > 0$, and it follows that $ef \rightarrow \{e_1, \ldots, e_n\}$.
	Since $F$ is a tight filter and $ef \in F$, it must be the case that $e_i \in F$ for some $i$. Now $e_i \leq ef$ and $e_i \leq sf$, so
	\[
		efe_i = e_i = sfe_i.
	\]
	Thus $[s, F] = [e, F]$ since $fe_i \in F$, and it follows that $\ker \Phi = \Gt(S)^{\scriptscriptstyle{(0)}}$. Hence $\Phi$ is injective.
	
	It remains to see that $\Phi$ is a homeomorphism. As in Lemma \ref{lem:uf-homeo}, it should be clear that $\Phi$ carries the basic open set $\Theta(s, D(e; {e_1, \ldots, e_n}))$
	in $\Gt(S)$ bijectively onto $\Theta(\tau(s), D(\tau(e); {\tau(e_1), \ldots, \tau(e_n)}))$. It follows that $\Phi$ is a homeomorphism, hence an isomorphism of topological groupoids.
\end{proof}

Having shown that $\Gt(S) \cong \Gt(S/\!\!\lr)$, we can now extend the characterization given by Exel and Pardo in \cite{ExelPardo} of the property that $\Gt(S)$ is effective
(though the term used there is \emph{essentially principal}). We first need a couple of other definitions from \cite{ExelPardo}. 

\begin{defn} Let $S$ be an inverse semigroup, and let $s \in S$. Given an idempotent $e \leq s^*s$ we say that:
\begin{enumerate}
\item $e$ is \emph{fixed} under $s$ if $se = e$, and
\item $e$ is \emph{weakly-fixed} under $s$ if $sfs^*f \neq 0$ for every nonzero idempotent $f \leq e$.
\end{enumerate}
\end{defn}

\begin{thm}
\label{thm:condL} 
Given an inverse semigroup $S$ with zero, consider the following statements:
\begin{enumerate}
	\item The tight groupoid of $S$ is effective.
	\item The standard action of $S$ on its tight filter space is topologically free.
	\item For every $s \in S$, and for every $e \in E$ that is weakly-fixed under $s$, there exists a finite cover for $e$ consisting of fixed idempotents.
	\item $S / \!\!\lr$ is fundamental.
\end{enumerate}
Then $(1) \Leftrightarrow (2) \Rightarrow (3) \Rightarrow (4)$. If $S$ is Hausdorff, then $(4) \Rightarrow (1)$ as well.
\end{thm}
\begin{proof}
	It was proved in \cite[Theorem 4.7]{ExelPardo} that $(1) \Leftrightarrow (2)$ and in \cite[Theorem 4.10]{ExelPardo} that $(2) \Rightarrow (3)$. We begin with the
	proof that $(3) \Rightarrow (4)$.

	Suppose (3) holds. We denote the equivalence class of an element $s \in S$ under $\lr$ by $[s]$. To prove that $S /\!\! \lr$ is fundamental, it suffices to 
	show that if $[s] \mcmu [e]$ for some idempotent $e$, then $[s] = [e]$. Since $[s] \mcmu [e]$ implies $[e] = [s^*s] = [es^*s]$, we may assume $e \leq s^*s$. Suppose 
	$[s] \mcmu [e]$ for some nonzero $s \in S$ and some idempotent $e \leq s^*s$. Note that $e \neq 0$ as $\lr$ is $0$-restricted. Then $sfs^* \lr ef$ for 
	every idempotent $f \in E$. In particular, for all nonzero $f \leq e$ we have $sfs^* \lr f$. Thus $sfs^*f \neq 0$, so $e$ is weakly-fixed under $s$. Therefore 
	there are idempotents $f_1, \dots, f_n$ such that $e \rightarrow \{f_1, f_2, \dots, f_n\}$ and $sf_i = f_i$ for each $i$. If we let $0 < t \leq s$, then
	\[
		ts^* = (st^*t)s^* \lr et^*t
	\]
	since $[s] \mcmu [e]$. Now $st^*t = t \neq 0$ implies $st^* \neq 0$, so $ts^* \neq 0$. It follows that $0 \neq et^*t \leq e$. Then $et^* tf_i \neq 0$ for some $i$. Notice that 
	$e t^* t f_i t = e t^* t f_i s t^* t = e t^* t f_i$. Thus $e t^* t f_i$ is a nonzero idempotent below both $t$ and $e$, so $s \rightarrow e$.

	Next let $0 < f \leq e$. Then there exists $f_i$ such that $f f_i \neq 0$. But then $(f f_i)s = f f_i \neq 0$, so $e \rightarrow s$. Hence $s \lr e$, and we have shown that 
	$[s] \mcmu [e]$ implies $[s] = [e]$. It follows that $S /\!\! \lr$ is fundamental. 

	Now assume $S$ is Hausdorff and $S /\!\! \lr$ is fundamental. Since $E(S / \!\!\lr)$ is $0$-disjunctive, it follows from \cite[Corollary 5.5]{LaLondeMilan} and Theorem 
	\ref{thm:groupoiddoublearrow} that $\Gt(S) \cong \Gt(S / \!\! \lr)$ is effective. 
\end{proof}

If $S$ is the inverse semigroup of a directed graph $\Lambda$, then Theorem \ref{thm:condL} shows that $S/\!\!\lr$ is fundamental if and only if $\Lambda$
satisfies Condition (L). In light of this observation, we now make the following definition.

\begin{defn}
	Let $S$ be an inverse semigroup with zero. We say that $S$ \emph{satisfies Condition (L)} if $S / \!\!\lr$ is fundamental.
\end{defn}

\section{Ideals and Groupoid Reductions}

In this section we analyze the ideals in an inverse semigroup $S$, and we begin our attempt to determine when (if ever) there are bijective correspondences between 
certain ideals in $S$ and invariant open subsets of the character spaces $\fE$ and $\tfE$. As a starting point, it is shown in \cite[Lemma 2.14]{BussExelMeyer} that 
ideals in $E(S)$ correspond to open subsets of the filter space $\fE$, where $\fE$ is equipped with a nonstandard (and non-Hausdorff) topology. Therefore, we should not expect a bijection in general, but it may still be possible to associate open sets in $\fE$ to ideals in $S$ in a meaningful way.

We should also note that the relationship between the ideal structure of an inverse semigroup and its groupoids has been studied by Lenz in \cite{lenz}. In particular, compare Theorem \ref{thm:tightideals} below with \cite[Lemma 7.7]{lenz}. We point out that \cite[Lemma 7.7]{lenz} and many related results in \cite{lenz} require that meets exist in $S$.

Let $S$ be an inverse semigroup with idempotent semilattice $E = E(S)$. 
By following the lead of \cite{BussExelMeyer}, we can use basic open sets of the form $N^e$ and $N_e$ to develop a relationship between ideals in $E$ and subsets 
of $\fE$.

\begin{defn}
	Let $A \subseteq \fE$ be a closed set. We define the \emph{kernel} of $A$ to be the set
	\[
		k(A) = \{ e \in E : e \not\in F \text{ for all } F \in A \}.
	\]
	Given an order ideal $X \subseteq E$, the \emph{hull} of $X$ is
	\[
		h(X) = \bigcap_{e \in X} N_e.
	\]
\end{defn}

\begin{rem}
	Notice that the hull of an ideal is automatically closed. Therefore, we could just as well associate an open set to $X$ by considering the complement 
	$\bigcup_{e \in X} N^e$. However, it will be more convenient and natural to work with closed sets for the time being.
\end{rem}

If $S$ does not contain a 0, then it is possible that $k(A)$ is empty for certain closed sets $A \subseteq \fE$. In particular, $k(\fE) = \emptyset$. We ultimately want $k(A)$ 
to be an ideal in $E$, so we will skirt this issue by assuming \emph{all inverse semigroups contain a 0 element}. If $S$ does not contain a 0, we will adjoin a trivial zero to it 
(i.e., we will replace $S$ with $S \cup \{0\}$). Note that the only effect on the ideal structure is the introduction of the ideal $\{0\}$.

The following proposition is known (it appears implicitly, for example, in \cite{SteinbergSimple}), but we include the short proof.

\begin{prop}
\label{prop:kernel}
	For any set $A \subseteq \fE$, 
	\[
		X = \{ e \in E : e \not\in F \text{ for all } F \in A \}
	\]
	is an ideal of $E$. Consequently, $A \mapsto k(A)$ defines a map from closed subsets of $\fE$ to ideals of $E$.
\end{prop}
\begin{proof}
	Let $e \in X$ and $f \in E$. Since $ef \leq e$, we must have $ef \not\in F$ for any filter $F$ that does not contain $e$. Hence $e \in X$ implies $ef \not\in F$ for all $F \in A$,
	or $ef \in X$. Therefore, $X$ is an ideal.
\end{proof}

We have already observed that the hull of an ideal is a closed set. However, we cannot realize every closed subset of $\fE$ in this way---we can only obtain sets whose
complements look like $\bigcup_{e \in X} N^e$ for some $X \subseteq E$. Obviously our open sets could be much more complicated, so the hull map $h$ 
will rarely be surjective. However, $h$ is always injective, so we have a correspondence between ideals in $E$ and certain closed sets in $\fE$.

\begin{prop}
	For all ideals $X \subseteq E$, we have $k(h(X)) = X$.
\end{prop}
\begin{proof}
	Let $X \subseteq E$ be an ideal, and note that
	\[
		k(h(X)) = \bigl\{ e \in E : e \not\in F \text{ for all } F \in h(X) \bigr\}.
	\]
	It is worth noting that
	\[
		h(X) = \bigl\{ F \in \fE : X \cap F = \emptyset \bigr\}.
	\]
	Therefore if $e \in X$, then $e \not\in F$ for all $F \in h(X)$. Hence $X \subseteq k(h(X))$. On the other hand, suppose $e \in k(h(X))$. Then $e \not\in F$ for all filters $F$
	satisfying $X \cap F = \emptyset$. Observe that if $e \not\in X$, then $e^\uparrow$ is a filter that does not meet $X$ (since $f \in {e^\uparrow} \cap X$ implies $e \in X$) but
	contains $e$, which is a contradiction. Therefore, $e \in X$, and we have $k(h(X)) \subseteq X$.
\end{proof}

\begin{rem}
	As we discussed above, it will not always be the case that $h(k(A)) = A$ for a closed set $A \subseteq \fE$. We always have
	\[
		h(k(A)) = \bigl\{ F \in \fE : F \cap k(A) = \emptyset \bigr\},
	\]
	from which it is easy to see that $A \subseteq h(k(A))$. However, the containment may be strict. For example, let $E$ be the semilattice depicted below,
	
	
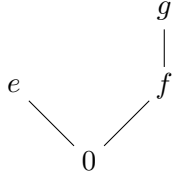
\begin{center}
	\begin{tikzpicture}
		\node at (0, 0) {$0$};
		\node at (1, 2) {$g$};
		\node at (-1, 1) {$e$};
		\node at (1, 1) {$f$};
		
		\draw (-0.2, 0.2) -- (-.8, .8);
		\draw (1, 1.25) -- (1, 1.75);
		\draw (0.2, 0.2) -- (.8, .8);
	\end{tikzpicture}
	\end{center}
	
	\noindent
	and consider the closed set $N^f = \{ \{f, g\} \}$. Then
	\[
		k(N^f) = \{e, 0\},
	\]
	so
	\[
		h(k(N^f)) = N_e = \bigl\{ \{f, g\}, \{g\} \bigr\},
	\]
	which is not equal to $N^f$.
	
	A more illuminating example is given by the tight spectrum $\tfE \subseteq \fE$. Notice that $\tfE$ is a closed invariant set, and that $k(\tfE) = \{0\}$ since every 
	nonzero idempotent belongs to an ultrafilter. However, $h(\{0\}) = \fE$, so $h(k(\tfE)) = \tfE$ if and only if every filter in $E$ is a tight filter. This observation suggests
	that we should eventually restrict our attention to the tight groupoid if we are to have any hope of obtaining a correspondence between ideals and closed sets of
	filters.
\end{rem}

As part of our investigation, we intend to relate ideals in $S$ to certain reductions of the universal groupoid $\G(S)$. To do so, we will need an appropriate notion of 
invariance for ideals in $E$.

\begin{defn} 
	An ideal $X \subseteq E$ is \emph{invariant} if $s s^* \in X$ implies $s^*s \in X$ for all $s \in S$.
\end{defn}

Notice that if $X$ is an invariant ideal in $E$, $s^*Xs \subseteq X$ for every $s \in S$. To see this, let $e = s^* f s$ for some $f \in X$ and let $x = fs$. Then 
$xx^* = f s s^* f \in X$, hence $e = s^* f s = x^* x \in X$. Also, it is easy to see that if $X \subseteq E$ is any ideal for which $s^* X s \subseteq X$ for all $s \in S$, then 
$X$ is invariant.

Since $S$ acts on the filter space $\fE$, we also have a notion of invariance for subsets of $\fE$. We say a set $A \subseteq \fE$ is \emph{invariant} precisely when it is 
invariant under this action, i.e., if $\beta_s(F) \in A$ whenever $F \in A$ and $F \in N^{s^*s}$.

\begin{prop} 
	An ideal $X \subseteq E$ is invariant if and only if $h(X)$ is an invariant set in $\fE$.
\end{prop}
\begin{proof}
	Suppose first that $X \subseteq E$ is an invariant ideal and let $F \in h(X)$. Let $s \in S$ with $s^*s \in F$. We need to show that $e \not\in \beta_s(F)$ for all $e \in X$.
	To the contrary, suppose that $\beta_s(F) \cap X \neq \emptyset$. Then there exist $e \in X$ and $f \in F$ such that $sfs^* \leq e$, which implies $sfs^* \in X$. 
	But then $s^*sf = s^*(sfs^*)s \in X \cap F$, contradicting the assumption that $F \in h(X)$. Therefore, $\beta_s(F) \cap X = \emptyset$, so $\beta_s(F) \in h(X)$ and $h(X)$ is invariant.

	Conversely, suppose that $h(X)$ is an invariant subset of $\fE$, and let $s \in S$ with $ss^* \in X$. Let $F$ be a filter containing $s^*s$. Then $\beta_s(F)$ 
	contains $ss^*$, so $\beta_s(F) \in h(X)^c$. Since $h(X)$ is invariant, $F = \beta_{s^*} (\beta_{s} (F))$ belongs to $h(X)^c$ as well. Thus $N^{s^*s} \subseteq h(X)^c$.
	In other words, $s^*s \not\in F$ for all $F \in h(X)$, so $s^*s \in k(h(X)) = X$. Hence $X$ is invariant.
\end{proof}

The next step in our analysis is to study ideals in $S$, which correspond precisely to invariant ideals in $E$. We omit the proof of the following well known fact.

\begin{prop} 
\label{prop:invariantideals}
	The map $X \mapsto SXS$ gives a one-to-one correspondence between invariant order ideals in $E$ and ideals in $S$, with inverse $I \mapsto I \cap E$.
\end{prop}
%
 
Our next goal is to show that $\G(I)$ and $\G(S/I)$ are isomorphic to the reductions of $\G(S)$ to the invariant sets $h(I\cap E)^c$ and $h(I \cap E)$, respectively. First we 
need to characterize the filter spaces of $I$ and $S/I$. Much of what follows is well known to experts and parts have appeared in other places (cf. \cite[Theorem 4.10]{MilanSteinberg}). We include the proofs for completeness.

\begin{prop}
\label{prop:filterhomeo}
	Let $I$ be an ideal of an inverse semigroup $S$.
	\begin{enumerate}
		\item The map $R_I(F) = F \cap I$ defines a homeomorphism of the open set $h(I \cap E)^c \subseteq \fE(S)$ onto $\fE(I)$.

		\item The map $Q_I(F) = F$ defines a homeomorphism of the closed set $h(I \cap E) \subseteq \fE(S)$ onto $\fE(S/I)$.		
	\end{enumerate}
\end{prop}
\begin{proof}
	It is not hard to see that if $F \in h(I \cap E)^c$, then $F \cap I$ is a filter in $I \cap E$. Suppose then that $F_1, F_2 \in \fE(S)$ with $F_1 \cap I = F_2 \cap I \neq \emptyset$. 
	Let $e \in F_1$. Then for any $f \in F_1 \cap I$, we have $ef \in F_1 \cap I = F_2 \cap I$. Since $ef \leq e$ and $F_2$ is a filter, it follows that $e \in F_2$. Hence 
	$F_1 \subseteq F_2$. A similar argument shows that $F_2 \subseteq F_1$, so $F_1 = F_2$ and $R_I$ is injective. It is clearly surjective, since for any $F \in \fE(I)$ the
	filter $F^\uparrow \in \fE(S)$ satisfies $R_I(F^\uparrow) = F$.
	
	It remains to see that $R_I$ is continuous and open. If $F \in h(I \cap E)^c$ and $e \in I\cap E$, then certainly $e \in F$ if and only if $e \in F \cap I$. It is then easy to 
	check that if $e, e_1, \ldots, e_n \in I\cap E$, then
	\[
		R_I^{-1} \bigl(N(e; {e_1, \ldots, e_n})\bigr) = N(e; {e_1, \ldots, e_n})  \cap h(I\cap E)^c.
	\]
	That is, preimages of basic open sets are open, so $R_I$ is continuous.
	
	Next we show that if $F \in h(I\cap E)^c$, then $e \in F$ if and only if $ef \in F \cap I$ for some $f \in F \cap I$. The forward direction is clear. Suppose then that there exists 
	$f \in F \cap I$ such that $ef \in F \cap I$. Then $ef \leq e$, so $e \in F$. Therefore, if $e, e_1, \ldots, e_n \in E(S)$, then we have $F \in N(e; {e_1, \ldots, e_n})$ if and only 
	if there exists $f \in F \cap I$ such that $F \cap I \in N({ef}; {e_1f, \ldots, e_nf})$. It follows that $R_I$ takes $N(e; {e_1, \ldots, e_n})$ onto the open set 
	$\bigcup_{f \in I\cap E} N({ef}; {e_1f, \ldots, e_nf}) \subseteq \fE(I)$. Therefore $R_I$ is an open map, hence a homeomorphism.
	
	Finally we consider the Rees quotient $S/I$. If $F \in \fE(S)$ with $F \cap I = \emptyset$, then $F \subseteq S \backslash I$, and it is easy to see that $F$
	is a filter in $E(S/I)$. Furthermore, it is evident that for all $F_1, F_2 \in h(I\cap E)$, we have $Q_I(F_1) = Q_I(F_2)$ if and only if $F_1 = F_2$. It is also easy to see that 
	$Q_I$ is surjective. Now notice that the relative basic open sets in $h(I\cap E)$ have the form $N(e; {e_1, \ldots, e_n})$ for idempotents $e, e_1, \ldots, e_n \in S \backslash I$. 
	For such a set we clearly have $F \in N(e; {e_1, \ldots, e_n})$ if and only if $Q_I(F) \in N({e}; {e_1, \ldots, e_n}) \subseteq \fE(S/I)$. Thus $Q_I$ is continuous and open,
	so it is a homeomorphism.
\end{proof}

\begin{thm}
\label{thm:universalreduction}
	Let $I$ be an ideal of an inverse semigroup $S$. 
	\begin{enumerate}
		\item The map $\Phi_I : \G(S) \vert_{h(I\cap E)^c} \to \G(I)$ given by 
			\[
				\Phi_I([s, F]) = [se, R_I(F)]
			\]
			for any $e \in F \cap I$ defines an isomorphism of topological groupoids.
		\item The map $\Phi^I : \G(S) \vert_{h(I\cap E)} \to \G(S/I)$ given by 
			\[
				\Phi^I([s, F]) = [s, q(F)]
			\]
			is an isomorphism of topological groupoids.
	\end{enumerate}
\end{thm}
\begin{proof}
	Suppose $[s, F] \in \G(S) \vert_{h(I\cap E)^c}$ and $e, f \in F \cap I = R_I(F)$. Then
	\[
		[se, R_I(F)] = [sef, R_I(F)] = [sf, R_I(F)],
	\]
	so $\Phi_I$ is well-defined.

	Now we set about checking that $\Phi_I$ is a groupoid homomorphism. As a first step, we need to know that the map $R_I$ from Proposition
	\ref{prop:filterhomeo} is equivariant with respect to the natural actions $\beta^S$ and $\beta^I$ of $S$ and $I$, respectively, on their filter spaces. Suppose 
	$F \in h(I\cap E)^c$ and let $s \in I$ with $s^*s \in F$. If $e \in \beta_s^S(F) \cap I$, $e \geq sfs^*$ for some $f \in F$. Then $sfs^* \in I$, so
	\[
		s^*(sfs^*)s = s^*s f s^*s = s^*s f
	\] 
	belongs to $F \cap I$. Hence
	\[
		s(s^*sf)s^* = (ss^*s) f s^* = sfs^*
	\]
	lies in $\beta_s^I(F \cap I)$, which implies that $e \in \beta^I_s(F \cap I)$. Thus $\beta_s^S(F) \cap I \subseteq \beta_s^I(F \cap I)$. On the other hand, suppose $e \in I\cap E$
	and $e \in \beta^I_s(F \cap I)$. Then $e \geq sfs^*$ for some $f \in F \cap I \subseteq F$, meaning that $e \in \beta_s^S(F)$. It follows that $\beta_s^I(F \cap I) \subseteq 
	\beta_s^S(F) \cap I$. It is now easy to see that
	\begin{align*}
		\Phi_I([t, \beta_s^S(F)]) \Phi_I([s, F]) &= [t, \beta_s^I(F \cap I)][s, F \cap I] \\
			&= [ts, F \cap I] \\
			&= \Phi_I([ts, F])
	\end{align*}
	so $\Phi_I$ is a homomorphism. Now suppose $\Phi_I([s, F]) = \Phi_I([t, F'])$. Then we have $[s, F \cap I] = [t, F' \cap I]$, so $F \cap I = F' \cap I$, which in turn 
	implies that $F = F'$ by Proposition \ref{prop:filterhomeo}. Moreover, there exists $e \in F \cap I \subseteq F$ such that $se = te$, so $[s, F] = [t, F]$. Hence
	$\Phi_I$ is injective. It is clearly surjective, so it implements an algebraic isomorphism between $\G(S) \vert_{h(I\cap E)^c}$ and $\G(I)$.

	Now suppose $\Theta^S(s, U)$ is a basic open set in $\G(S) \vert_{h(I\cap E)^c}$, where $s \in I$ and $U \subseteq \fE(S)$ is an open set satisfying $U \subseteq h(I\cap E)^c$.
	Then it is not hard to see that
	\[
		\Phi_I(\Theta^S(s, U)) = \bigr\{ [s, F \cap I] \in \G(I) : F \cap I \in R_I(U) \bigr\} = \Theta^I(s, R_I(U)),
	\]
	which is a basic open set in $\G(I)$. Therefore, $\Phi_I$ takes basic open sets in $\G(S) \vert_{h(I\cap E)^c}$ to basic open sets in $\G(I)$, so it is a homeomorphism.

	Now we turn to $\Phi^I$. Note that $F \in h(I\cap E)$ and $s^*s \in F$ together imply that $s \in S \backslash I$, so the definition of $\Phi^I$ makes sense. Now we claim 
	that for all $F \in h(I\cap E)$ and $s \in S$ with $s^*s \in F$, we have $Q_I(\beta_s^S(F)) = \beta_{s}^{S/I}(Q_I(F))$. First let $e \in \beta_s^S(F)$,
	so $e \geq sfs^*$ for some $f \in F$. Note that $sfs^* \not\in I$ (since $I \cap E$ is invariant), so $e \not\in I$ and we still have $e \geq sfs^*$ in $S/I$.
	In other words, $e \in \beta_{s}^{S/I}(Q_I(F))$. The other containment is done similarly. Hence the map $Q_I : h(I\cap E) \to \fE(S/I)$ from Proposition \ref{prop:filterhomeo} 
	is equivariant. It is now easy to see that
	\begin{align*}
		\Phi^I([t, \beta_s^S(F)]) \Phi^I([s, F]) &= [t, \beta_{q(s)}^{S/I}(Q_I(F))][s, Q_I(F)] \\
			&= [ts, Q_I(F)] \\
			&= \Phi^I([ts, F]) \\
			&= \Phi^I([t, \beta_s^S(F)][s, F]),
	\end{align*}
	so $\Phi^I$ is a homomorphism. Next observe that if $\Phi^I([s, F]) = \Phi^I([t, F'])$, then $[s, Q_I(F)] = [t, Q_I(F')]$, so $F = F'$ since $Q_I$ is injective. Moreover,
	there exists $e \in Q_I(F) = F$ such that $se = te$. Hence $[s, F] = [t, F]$, so $\Phi^I$ is injective. It is clearly surjective. Finally, it is easy to see that if $\Theta^S(s, U)$ 
	is a basic open set in $\G(S) \vert_{h(I\cap E)}$, then $\Phi^I(\Theta^S(s, U)) = \Theta^{S/I}(s, Q_I(U))$, so $\Phi^I$ is a homeomorphism.
\end{proof}

\subsection{Invariant order ideals and the tight filter space of $E$}

We saw in Proposition \ref{prop:invariantideals} that ideals in $S$ correspond to invariant order ideals in $E$, which in turn correspond to certain closed invariant sets of 
filters. However, it is not the case that every closed invariant set in $\fE$ arises from an invariant order ideal. 

The situation is much nicer in the tight filter space of $E$, at least if we make a couple of finiteness assumptions on $E$. Under such assumptions, we are able to develop 
a correspondence between open invariant sets of tight filters and certain invariant order ideals in $E$. 
We need a few definitions in order to lay out the appropriate hypotheses. Following Exel and Pardo, an {\it outer cover} for an order ideal $X \subseteq E$ is a set $C \subseteq E$ such 
that for every nonzero $f \in X$ there exists $c \in C$ such that $cf \neq 0$. If $C$ is an outer cover for $X$ with $C \subseteq X$ then $C$ is called a {\it cover} for $X$. 
Note that $e \covby \{e_1, e_2, \dots, e_n\}$ if and only if $\{e_1, e_2, \dots, e_n\}$ is a cover for $e^\down$. Also, for $f \in E$ we define $f^{\perp} = \{e \in E: ef = 0\}$. The following definition was introduced by Lenz in \cite{lenz}.

\begin{defn} A meet semilattice $E$ has the \emph{trapping condition} if for all $e,f \in E \setminus \{0\}$ where $f<e$, there exists $e_1, e_2, \dots, e_m$ such that 
\[
e \covby \{e_1, \dots, e_m, f\} \text{ and } e_i \in e^\down \cap f^{\perp}.
\] 
\end{defn}

Note in the above definition it is possible that $e^\down \cap f^{\perp} = \{0\}$, in which case $e \covby f$ and the semilattice fails to be $0$-disjunctive. 
Thus our definition of the trapping condition is slightly more general than the one given in \cite{LawsonCompactable}.

\begin{defn} Let $S$ be an inverse semigroup. We say that $E(S)$ \emph{admits finite covers} if for each invariant order ideal $X$ of $E$ there exists a finite cover $C$ of $X$.
\end{defn}

\begin{exmp}\label{exmp:directedgraphs} For the inverse semigroup $\mcS_{\Lambda}$ of a finite graph $\Lambda$, the semilattice admits finite covers and satisfies the trapping condition. Suppose that $X$ is an invariant order ideal of the semilattice $E(\mcS_{\Lambda}) = \{(\alpha,\alpha) : \alpha \in \Lambda^*\}$. Notice that
\[(\alpha, s_{\alpha})(\alpha, s_{\alpha})^* = (\alpha, \alpha) \in X \text{ implies } (s_{\alpha}, s_{\alpha}) = (\alpha, s_{\alpha})^*(\alpha, s_{\alpha}) \in X.
\] 
Thus $X$ is generated by the idempotents $\{ (s_{\alpha},s_{\alpha}) : (\alpha,\alpha) \in X\}$. Since there are finitely many vertices, $S_{\Lambda}$ admits finite covers. The proof that $S_{\Lambda}$ has the trapping condition relies only on the fact that $\Lambda$ is \emph{row-finite}, that is each vertex in $\Lambda$ has finite in-degree. Let $(\beta, \beta) < (\alpha, \alpha)$. Then $\beta$ is an extension of $\alpha$. That is, $\beta = \alpha \alpha'$ for some path $\alpha'$. If $(\gamma, \gamma) < (\alpha, \alpha)$, then either $(\gamma, \gamma) \leq (\beta, \beta)$ or $\gamma$ is an extension of a path $\alpha \alpha_f' f$ where $f$ is an edge with range on $\alpha'$, $\alpha_f'$ is a subpath of $\alpha'$, and $(\alpha \alpha_f' f, \alpha \alpha_f' f)(\beta, \beta) = 0$. There are finitely many such edges, call them $\{f_1, \dots, f_m\}$ and let $e_i = (\alpha \alpha_{f_i}' f_i, \alpha \alpha_{f_i}' f_i)$ for each $i$. Then $e_i (\beta, \beta) = 0$ for each $i$ and
\[ 
(\alpha, \alpha) \covby \{e_1, e_2, \dots, e_m, (\beta, \beta)\}.
\]
Thus $S_{\Lambda}$ satisfies the trapping condition.
\end{exmp}

\begin{lem}
\label{lem:emptybasic} 
	Let $E$ be a meet semilattice, and suppose $e \in E$ and $e_i \leq e$ for $1 \leq i \leq n$. Then $D(e; e_1,e_2, \dots, e_n) = \emptyset$ if and only if 
	$e \covby  \{e_1, e_2, \dots, e_n\}$. Moreover, if $E$ satisfies the trapping condition and $F \in D(e; e_1,e_2, \dots, e_n)$, then there exists $x \in F$ such that 
	$x \leq e$ and $xe_i = 0$ for $1 \leq i \leq n$.
\end{lem}
\begin{proof}
	Suppose $D(e; e_1,e_2, \dots, e_n) = \emptyset$ and let $0 \neq x \leq e$. Then there exists a tight filter $F$ containing $x$ and hence $e$. Since 
	$D(e; e_1,e_2, \dots, e_n) = \emptyset$, we have $e_i \in F$ for some $i$. Thus $x e_i \neq 0$, so $e \covby  \{e_1, e_2, \dots, e_n\}$. Conversely, suppose 
	$e \covby  \{e_1, e_2, \dots, e_n\}$. If $F$ is a tight filter containing $e$, then $e_i \in F$ for some $i$. Hence $D(e; e_1,e_2, \dots, e_n) = \emptyset$.

	Finally, assume $E$ has the trapping condition and $F \in D(e; e_1,e_2, \dots, e_n)$ for some tight filter $F$. For each $i$, we have $e_i < e$ and there exist 
	$\{x_1, \dots, x_m\}$ with $x_j \in\, e^\down \cap e_i^\perp$ for $1 \leq j \leq m$ and $e \covby \{e_i, x_1, \dots, x_m\}$. Then $F \not\in D(e; e_i,x_1, \dots, x_m)$, 
	so there exists $f_i \in \{x_1, \dots, x_m\}$ such that $f_i \in F$ and $f_i e_i = 0$. If $x = f_1 f_2 \cdots f_n$, then $x \in F$, $x \leq e$, and $xe_i = 0$ for all $i$.
\end{proof}

\begin{defn} 
	An ideal $X$ of $E$ is {\it saturated} if $e \covby \{e_1, \dots, e_n \}$ and $e_i \in X$ for $1 \leq i \leq n$ implies $e \in X$. We say an ideal $I$ of $S$ is 
	{\it saturated} if $I \cap E$ is saturated.
\end{defn}

One can quickly prove that an ideal $I$ of $S$ is saturated if and only if $s \covby \{s_1, \dots, s_n \}$ and $s_i \in I$ for $1 \leq i \leq n$ implies $s \in I$. We now show that there is a correspondence between closed invariant sets in the space of tight filters and saturated invariant ideals of idempotents.

\begin{prop} 
	Suppose $E$ is a meet semilattice and let $A$ be a closed subset of $\tfE$. Then $k(A)$ is a saturated ideal in $E$.
\end{prop}
\begin{proof}
	We already know from Proposition \ref{prop:kernel} that $k(A)$ is an ideal in $E$. Suppose $e \covby \{e_1, \dots, e_n \}$, where $e_i \in k(A)$. If $e \in F$ for some 
	$F \in A$, then $F \in D(e; e_1,e_2, \dots, e_n)$, contradicting the fact that $D(e; e_1,e_2, \dots, e_n)$ is empty. Thus for all 
	$F \in A$, $e \not \in F$. In other words, $e \in k(A)$, hence $k(A)$ is saturated.
\end{proof}

Recall that if $X \subseteq E$ is an invariant order ideal, then $h(X)$ is a closed invariant subset of $\fE$. Consequently, $\htt(X) = h(X) \cap \tfE$ is a closed invariant
subset of $\tfE$. However, the map $X \mapsto \htt(X)$ need not be injective (unlike for $\fE$), since $k(\htt(X))$ is necessarily a saturated ideal.

\begin{thm} 
\label{thm:tightideals}
	Let $S$ be an inverse semigroup with semilattice $E$. Suppose that $E$ satisfies the trapping condition and admits finite covers. Then there is a correspondence between saturated invariant ideals in $E$ and closed invariant subsets of $\tfE$. 
\end{thm}
\begin{proof}
	First we let $X$ be a saturated invariant ideal in $E$ and we show that $X = k(\htt(X))$. Let $e \in X$. Then $e \not \in F$ for all $F \in \htt(X)$, so $e \in k(\htt(X))$.
	Conversely, suppose $e \in k(\htt(X))$. Then for any tight filter $F$ containing $e$, $F \cap X \neq \emptyset$. Let $\{x_1, \dots, x_n\}$ be a finite cover for 
	$X$, and suppose $F \in D(e; ex_1, \dots, ex_n)$. Given $x \in F \cap X$, we claim that $ex \covby \{e x x_1, \dots, e x x_n\}$. To see it, suppose $f \leq ex$. Then 
	$f \leq x$ and there exists $x_i$ such that $f x_i \neq 0$, so $f(e x x_i) = (fex) x_i = fx_i \neq 0$. Thus $ex \in F$ implies that $exx_i \in F$ for some $i$, hence 
	$ex_i \in F$. This contradicts the assumption that $F \in D(e; ex_1, \dots, ex_n)$. Therefore, $D(e; ex_1, \dots, ex_n) = \emptyset$. By Lemma \ref{lem:emptybasic}, $e \covby \{ex_1, \dots, ex_n\}$. Since $X$ is saturated, $e \in X$.

	Next we show that $\htt(k(A)) = A$ for any closed set of tight filters $A$. Let $F \in A$ and suppose that $e \in k(A)$. Then $e \not \in F$, so $F \in \htt(k(A))$. 
	Conversely, let $F \in \htt(k(A))$. Choose a basic open set $D(e; e_1,e_2, \dots, e_n)$ 
	containing $F$. Since $D(e; e_1,e_2, \dots, e_n)$ is nonempty, there exists $x \in F$ such that $x \leq e$ and $x e_i = 0$ for $1 \leq i \leq n$ by Lemma 
	\ref{lem:emptybasic}. Since $x \in F$, there exists $G \in A$ such that $x \in G$. Then $e \in G$, since $x \leq e$. Also $e_i \not \in G$ for $1 \leq i \leq n$ since 
	$x e_i = 0$. Thus $G \in D(e; e_1,e_2, \dots, e_n)$. Therefore, any open set containing $F$ intersects $A$ nontrivially. Since $A$ is closed, $F \in A$.
\end{proof}

We can also write down analogues of Proposition \ref{prop:filterhomeo} and Theorem \ref{thm:universalreduction} for the tight groupoid of $S$.

\begin{prop}
\label{prop:tightfilterhomeo}
	Let $I$ be a saturated ideal of an inverse semigroup $S$.
	\begin{enumerate}
		\item The map $R_I : h(I \cap E)^c \to \fE(I)$ given by $R_I(F) = F \cap I$ restricts to a homeomorphism of $h(I \cap E)^c \cap \tfE(S)$ onto $\tfE(I)$.

		\item The map $Q_I : h(I \cap E) \to \fE(S/I)$ given by $Q_I(F) = F$ restricts to a homeomorphism of $\htt(I \cap E) \subseteq \tfE(S)$ onto $\tfE(S/I)$.		
	\end{enumerate}
\end{prop}
\begin{proof}
	Since $R_I$ is a homeomorphism by Proposition \ref{prop:filterhomeo}, it is sufficient to check that $R_I$ maps $h(I \cap E)^c \cap \tfE(S)$ onto $\tfE(I)$. Suppose
	$F \in \tfE(S)$ and $F \cap I \neq \emptyset$. Let $e \in F \cap I$, and suppose $e \rightarrow \{e_1, e_2, \ldots, e_n\}$, where $e_i \in I$ for 
	$1 \leq i \leq n$. Since $F$ is a tight filter, we must have $e_i \in F$ for some $i$. Thus $e_i \in F \cap I$, and it follows that $F \cap I$ is a tight filter in $I \cap E$.
	
	Now suppose $F \in \tfE(I)$. We observed in the proof of Proposition \ref{prop:filterhomeo} that $R_I^{-1}(F) = F^\up$, and we claim that $F^\up$ is a tight filter. Let 
	$e \in F^\up$, and suppose $e \covby \{e_1, e_2, \ldots, e_n\}$. If we choose $f \in F$ with $f \leq e$, then it is easy to
	check (see the proof of Theorem \ref{thm:tightideals}) that $f \covby \{e_1f, e_2f, \ldots, e_nf\}$. Note that $e_i f \in I \cap E$ for $1 \leq i \leq n$, so
	$ef \in I \cap E$ since $I$ is saturated. Since $F$ is a tight filter, $e_i f \in F$ for some $i$, whence $e_i \in F^\up$. Therefore, $F^\up \in \tfE(S)$. It is now
	evident that $F \mapsto F \cap I$ defines a homeomorphism of $\tfE(S) \backslash \htt(I \cap E)$ onto $\tfE(I)$.
	
	As with $R_I$, we already know that $Q_I$ is a homeomorphism, and we simply need to check that $Q_I(\htt(I \cap E)) = \tfE(S/I)$. Suppose first that $F \in \htt(I \cap E)$.
	Let $e \in Q_I(F)$, and suppose $e \covby \{e_1, e_2, \ldots, e_n\}$ where $e_i \in E(S/I)$ for $1 \leq i \leq n$. We are effectively saying that
	$e \covby \{e_1, e_2, \ldots, e_n\}$ in $E(S)$, so some $e_i \in F$ since $F$ is a tight filter. Hence $Q_I(F) \in \tfE(S/I)$.
	
	Now suppose $F \in \tfE(S/I)$. Let $e \in Q_I^{-1}(F) = F$, and suppose $e \covby \{e_1, e_2, \ldots, e_n\}$ with $e_i \in E$ for $1 \leq i \leq n$. Since $I$ is 
	saturated, it cannot be the case that $e_i \in I \cap E$ for all $i$, as this would force $e \in I$ and $F \cap I = \emptyset$. After possibly relabeling, we may assume 
	$e_1, \ldots, e_m \not\in I$ (where $1 \leq m \leq n$) and $e_{m+1}, \ldots, e_n \in I$ (where it is possible that $m=n$ and no $e_i$ belongs to $I$). We claim that 
	$e \covby \{e_1, e_2, \ldots, e_m\}$ in $E(S/I)$. To see this, suppose $x \in S\backslash I$ with $x \leq e$. Then $x \covby \{e_1x, e_2x, \ldots, e_nx\}$. If 
	$e_i x = 0$ for $1 \leq i \leq m$, then in fact $x \covby \{e_{m+1} x, \ldots, e_nx\}$. But $e_i x \in I$ for $m+1 \leq i \leq n$, which forces $x \in I$ since $I$ is 
	saturated. This is a contradiction, so we must have $xe_i \neq 0$ for some $1 \leq i \leq m$. Hence $e \covby \{e_1, e_2, \ldots, e_m\}$ in $E(S/I)$, so $e_i \in F$ 
	for some $i$ since $F$ is a tight filter in $E(S/I)$. It follows that $Q_I^{-1}(F) \in \tfE(S)$, so $Q_I$ restricts to a homeomorphism from $\htt(I)$ onto $\tfE(S/I)$.
\end{proof}

\begin{thm}
\label{thm:tightreduction}
	Let $I$ be a saturated ideal of an inverse semigroup $S$. 
	\begin{enumerate}
		\item The map $\Phi_I : \Gt(S) \vert_{\tfE(S) \backslash \htt(I \cap E)} \to \Gt(I)$ given by 
			\[
				\Phi_I([s, F]) = [se, R_I(F)]
			\]
			for any $e \in F \cap I$ defines an isomorphism of topological groupoids.
		\item The map $\Phi^I : \Gt(S) \vert_{\htt(I \cap E)} \to \Gt(S/I)$ given by 
			\[
				\Phi^I([s, F]) = [s, q(F)]
			\]
			is an isomorphism of topological groupoids.
	\end{enumerate}
\end{thm}
\begin{proof}
	Note that $\Phi_I$ and $\Phi^I$ are simply the restrictions of the corresponding maps in Theorem \ref{thm:universalreduction}. Thus Proposition \ref{prop:tightfilterhomeo}
	and Theorem \ref{thm:universalreduction} together guarantee that $\Phi_I$ maps the reduction of $\Gt(S)$ to ${\tfE(S) \backslash \htt(I \cap E)}$ bijectively onto a 
	subgroupoid of $\G(I)$ with unit space $\tfE(I)$. In other words, the range of $\Phi_I$ is contained in $\Gt(I)$. It is not hard to see that $\Phi_I$ is surjective, so it is an 
	isomorphism. The argument for $\Phi^I$ is similar.
\end{proof}

\section{Condition (K) for Inverse Semigroups}

In this section we define Condition (K) for inverse semigroups. We also investigate the relationship between saturated ideals of $S$ and the ideals in both the tight 
$C^*$-algebra and the Steinberg algebra associated to $S$. 

We first review Condition (K) and its consequences for directed graphs. A row-finite directed graph $\Lambda$ is said to satisfy \emph{Condition (K)} if and only if no vertex 
in $\Lambda$ is the base point of exactly one simple loop. It is observed in \cite[Corollary 6.5]{CE-MHS} that $\Lambda$ satisfies Condition (K) if and only if its path groupoid
is strongly effective. Alternatively, it is shown in \cite[Lemma 4.7]{Raeburn} that $\Lambda$ satisfies Condition (K) if and only if $\Lambda / H$ satisfies Condition (L) for every 
saturated hereditary set $H$ of vertices. (Recall that $H$ is \emph{hereditary} if $r_{e} \in H$ implies $s_{e} \in H$ for all $e \in \Lambda^1$, and $H$ is \emph{saturated} 
if $r^{-1}(v) \neq \emptyset $ and $\{ s_{e} : r_{e} = v\} \subseteq H$ imply $v \in H$.)

For any ideal $\mcI$ of $C^*(\Lambda)$, the set of vertices in $\mcI$, denoted $H_\mcI$, is saturated and hereditary. One can then consider the quotient graph 
$\Lambda / H_{\mcI}$ formed by removing $H_{\mcI}$ and any edges incident to those vertices \cite[Chapter 4]{Raeburn}. Provided $\Lambda / H_{\mcI}$ satisfies 
Condition (L), the Cuntz-Krieger uniqueness theorem implies that $C^*(\Lambda) / \mcI$ is isomorphic to $C^*(\Lambda / H_{\mcI})$. Therefore, if $\Lambda$
satisfies Condition (K), the ideals of $C^*(\Lambda)$ correspond to the saturated, hereditary sets of vertices, and the quotient by any ideal $\mcI$ is isomorphic to 
$C^*(\Lambda / H_{\mcI})$.

\begin{defn} 
	Let $S$ be an inverse semigroup with $0$. We say that $S$ \emph{satisfies Condition (K)} if and only if $S/I$ satisfies Condition (L) for each saturated ideal $I$ of $S$.
\end{defn}

Ideally one would like to state the above definition in a way that is intrinsic to $S$. We conjecture that such a definition would be that $S$ is fundamental and $E(S)$ satisfies 
some additional property. We have not yet been able to give such a definition, but we will develop a nice sufficient condition for $S$ to satisfy Condition (K): that all congruences 
on $S$ are Rees congruences. We characterize this sufficient condition in Theorem \ref{thm:Rees} below.

First we show that Condition (K) is related to the property that $\Gt(S)$ is strongly effective, as in the case of the path groupoid of a directed graph.

\begin{thm}\label{thm:condK}
	Let $S$ be a Hausdorff inverse semigroup with $0$, and suppose $E(S)$ has the trapping condition. Then $\Gt(S)$ is strongly effective if and only if $S$ satisfies 
	Condition (K).
\end{thm}
\begin{proof} 
	Suppose $\Gt(S)$ is strongly effective, and let $I$ be a saturated ideal of $S$. Then $\Gt(S/I) \cong \Gt(S) \vert_{\htt(I\cap E)}$ by Theorem \ref{thm:tightreduction}. 
	Thus $\Gt(S/I)$ is effective. It follows from Theorem \ref{thm:condL} that $S/I$ satisfies Condition (L).

	Now assume $S$ satisfies Condition (K), and suppose $A$ is a closed invariant subset of $\Gt(S)^{(0)}$. Let $I = S k(A) S$ be the saturated ideal of $S$ generated by 
	$k(A)$. Then by the proof of Theorem \ref{thm:tightideals}, $A = \htt(k(A)) = \htt(I \cap E)$. (Note that this part of the proof does not require $S$ to admit finite covers.) 
	Therefore $\Gt(S) \vert_{A} \cong \Gt(S/I)$ by Theorem \ref{thm:tightreduction}. Thus $\Gt(S) \vert_{A}$ is effective, and it follows that $\Gt(S)$ is strongly effective.
\end{proof}

\begin{rem}
	Since the inverse semigroup of any row-finite directed graph $\Lambda$ satisfies the trapping condition, it follows from Theorem \ref{thm:condK} that $S_\Lambda$ 
	satisfies our version of Condition (K) if and only if $\Lambda$ satisfies Condition (K).
\end{rem}

Our goal now is to describe the ideals in $C^*(\Gt(S))$ and the basic ideals in the Steinberg algebras $A_R(\Gt(S))$, where $R$ is any commutative ring with identity, in 
terms of saturated ideals in $S$. Our analysis involves a direct application of the main results of \cite{BCFS} and \cite{CE-MHS}. Since we will not work directly with 
Steinberg algebras in this paper, we refer the reader to those papers for definitions.

Our first result in this direction is a direct translation of \cite[Theorem 3.1]{CE-MHS} to the context of the present paper.

\begin{thm} Let $S$ be a Hausdorff inverse semigroup with $0$, and let $R$ be a commutative ring with identity. Suppose $E(S)$ has the trapping condition and admits finite covers. Then $S$ satisfies Condition (K) if and only if
\[
	I \mapsto \mcJ_{I} = \bigl\{ f \in A_R(\Gt(S)) : \supp f \subseteq \Gt(S) \vert_{\htt(I\cap E)^c} \bigr\}
\]
is a lattice isomorphism from the saturated ideals of $S$ onto the basic ideals of the Steinberg algebra $A_R(\Gt(S))$.
\end{thm}

\begin{proof}
It follows from Theorem \ref{thm:tightideals} that the map
\[
I \mapsto \htt(I\cap E)^c = \bigcup_{e \in I \cap E} D^e
\]
defines a bijection between saturated ideals in $S$ and open invariant subsets of $\Gt(S)^{(0)}$. One can quickly verify that this map is also a lattice homomorphism where we view both sets as lattices under set inclusion, intersection, and union. To finish the proof we note that the map $U \mapsto \Gt(S)\vert_{U}$ is a lattice isomorphism by \cite[Theorem 3.1]{CE-MHS}.
\end{proof} 

We can also characterize simplicity for the Steinberg algebra $A_{\bbC}(\Gt(S))$ in terms of the structure of $S$. The following is a special example \cite[Theorem 4.1]{BCFS} in the present context.

\begin{cor} Let $S$ be a Hausdorff inverse semigroup with $0$. Suppose $E(S)$ has the trapping condition and admits finite covers. Then the Steinberg algebra $A_{\bbC}(\Gt(S))$ is simple if and only if $S$ satisfies Condition (L) and contains no nontrivial saturated ideals.
\end{cor}

Similarly, there is a description of the lattice of ideals in $C^*(\Gt(S))$ for an amenable, Hausdorff $S$ satisfying the same finiteness conditions. The following is a translation of \cite[Corollary 5.9]{BCFS} to the tight $C^*$-algebra of $S$.

\begin{thm} Let $S$ be a Hausdorff, amenable inverse semigroup with $0$. Suppose $E(S)$ has the trapping condition and admits finite covers. Then $S$ satisfies Condition (K) if and only if
\[
	I \mapsto \mcJ_{I} = \overline{\bigl\{ f \in C_{c}(\Gt(S)) : \supp f \subseteq \Gt(S) \vert_{\htt(I\cap E)^c} \bigr\}}
\]
is a lattice isomorphism from the saturated ideals of $S$ onto the ideals of $C^*(\Gt(S))$. 
\end{thm}

Finally, we have conditions that guarantee simplicity for $C^*(\Gt(S))$ based on \cite[Theorem 5.1]{BCFS}.

\begin{cor} Let $S$ be a Hausdorff, amenable inverse semigroup with $0$. If $E(S)$ has the trapping condition and admits finite covers, then $C^*(\Gt(S))$ is simple if and only if $S$ satisfies Condition (L) and contains no nontrivial saturated ideals.
\end{cor}

It can be difficult to show that an inverse semigroup satisfies Condition (K), since the condition involves every quotient of $S$ by a saturated ideal. We now develop a sufficient 
condition that is easier to check, namely that all congruences on $S$ are Rees congruences. To show that this condition implies Condition (K) we characterize inverse semigroups 
having this property as those for which each quotient $S/I$ is fundamental and has a $0$-disjunctive semilattice of idempotents. The proof of the following proposition appears in \cite[Lemma IV.3.10]{PetrichBook}.
\begin{prop}
Let $S$ be an inverse semigroup with $0$. Then every $0$-restricted congruence on $S$ is equality if and only if $S$ is fundamental and $E(S)$ is $0$-disjunctive.
\end{prop}

%

\begin{thm}
\label{thm:Rees}
Let $S$ be an inverse semigroup with $0$. Then every congruence on $S$ is a Rees congruence if and only if $S/I$ is fundamental and has a $0$-disjunctive semilattice of idempotents for every ideal $I$ of $S$. 
\end{thm}

\begin{proof}
	First suppose $S$ is an inverse semigroup with $0$ for which every congruence is a Rees congruence. Then any congruence $\rho$ is equal to the Rees congruence 
	$\rho_{J}$ where $J = \rho^{-1}(0)$ is the $\rho$ class of $0$. Let $I$ be an ideal of $S$. By the previous proposition it suffices to prove that every $0$-restricted 
	congruence on $S/I$ is equality. Let $\pi_I : S \to S/I$ be the quotient map and suppose that $\widehat{\rho}$ is a $0$-restricted congruence on $S/I$. Then 
	$\rho = \{(s,t): \pi_{I}(s)  \mathrel{\widehat{\rho}}  \pi_I(t)\}$ is a congruence on $S$. Thus $\rho$ is a Rees congruence, and so $s \mathrel{\rho} t$ if and only if $s=t$ 
	or $s,t \in \rho^{-1}(0)$. Now suppose $\pi(s) \mathrel{\widehat{\rho}} \pi(t)$ for some $s \neq t$. Then $s,t \in \rho^{-1}(0) \subseteq I$ and $\pi_I(s) = \pi_I(t)$. Thus 
	$\widehat{\rho}$ is equality.

	Now suppose $S$ is an inverse semigroup with $0$ and that $S/I$ is fundamental and $0$-disjunctive for every ideal $I$ of $S$. Let $\rho$ be a congruence on $S$. 
	Put $I = \rho^{-1}(0)$ and consider the congruence $\widehat{\rho}$ on $S/I$ defined by $[s] \mathrel{\widehat{\rho}} [t]$ if and only if $s \mcrho t$. By the definition of $I$, 
	$\widehat{\rho}$ is $0$-restricted. Therefore it must be equality. Thus $s \mcrho t$ if and only if $s = t$ or $s,t \in \rho^{-1}(0)$, so $\rho$ is a Rees congruence.
\end{proof}

\begin{cor}
	Let $S$ be an inverse semigroup with $0$. If every congruence on $S$ is Rees then $S$ satisfies Condition (K).
\end{cor}
\begin{proof}
	Suppose that every congruence on $S$ is Rees and let $I$ be a saturated ideal of $S$. Then $S/I$ is fundamental and has a $0$-disjunctive semilattice of idempotents. 
	By (3) of Proposition \ref{prop:0disjunctive} it follows that $\leftrightarrow$ is equality on $S/I$ and hence $S/I$ satisfies Condition (L).
\end{proof}

An inverse semigroup is said to be \emph{congruence free} if every congruence on $S$ is equality or the trivial congruence. As a corollary of Theorem \ref{thm:Rees}, we 
have the following characterization of the property that an inverse semigroup is congruence free. It was originally proved by Baird \cite[Theorem 3]{Baird} that a fundamental, 
$0$-simple inverse semigroup with $0$-disjunctive semilattice is congruence free.

\begin{cor}
	Let $S$ be an inverse semigroup with $0$. Then $S$ is congruence free if and only if $S$ is fundamental, $0$-simple, and has a $0$-disjunctive semilattice of idempotents.
\end{cor}

\begin{proof}

If $S$ is congruence free, then $S$ is fundamental and $\leftrightarrow$ is equality since $\mu$ and $\leftrightarrow$ are $0$-restricted congruences. By Proposition \ref{prop:0disjunctive}, it follows that $E(S)$ is $0$-disjunctive. Also, $S$ is $0$-simple since for any ideal $I$, the congruence $\rho_I$ is either trivial, in which case $I = S$, or equality, in which case $I = \{0\}$.

The converse follows directly from Theorem \ref{thm:Rees}.
\end{proof}

\section{Self-Similar Graph Actions}
In the previous section we discussed the ideals in the tight $C^*$-algebra of an inverse semigroup satisfying Condition (K). In this section we will highlight these results for the recently defined inverse semigroups associated with self-similar graph actions \cite{ExelPardoGraph}. Rather than dealing with Condition (K) directly, we focus on the property that all congruences are Rees congruences. It turns out that $S_{\Lambda}$ has only Rees congruences exactly when $\Lambda$ satisfies a condition discovered by Mesyan and Mitchell \cite{MesyanMitchell} that we will call Condition (M). In this section we develop conditions on a self-similar graph action $(G,\Lambda,\varphi)$ that are equivalent to the property that all congruences on $\mcS_{G,\Lambda}$ are Rees. Along the way we describe a correspondence between Rees ideals of $\mcS_{G,\Lambda}$ and certain sets of vertices in $E$. This is similar to the ideal theory developed for graph algebras in terms of saturated, hereditary sets of vertices. 

Consider the inverse semigroup $\mcS_{G,\Lambda}$ of a self-similar graph action $(G,\Lambda,\varphi)$ of a countable discrete group $G$ on a finite directed graph $\Lambda$ with no sources or sinks, as defined in \cite{ExelPardoGraph}. By definition $\varphi: G\times \Lambda^1 \to G$ satisfies
\[
\varphi(g, e)x = gx, \, \text{ for all } g\in G, \, e\in \Lambda^1, \, x\in \Lambda^0
\]
Then the action of $G$ and the map $\varphi$ extend to $\Lambda^*$ in a natural way. For every $g, h\in G$, for every $x\in \Lambda^0$ and for every $\alpha$ and $\beta$ in $\Lambda^*$ such that $s_{\alpha} = r_{\beta}$ we have

\begin{tabular}{p{6cm} p{6cm}}
\begin{enumerate}
\item[(E1)] $(gh)\alpha = g(h\alpha)$
\item[(E2)] $\varphi(gh, \alpha) = \varphi(g, h\alpha)\varphi(h, \alpha)$
\item[(E3)] $\varphi(g, x) = g$
\item[(E4)] $r_{g\alpha} = gr_{\alpha}$
\end{enumerate}
&
\begin{enumerate}
\item[(E5)] $s_{g\alpha} = gs_{\alpha}$
\item[(E6)] $ \varphi(g, \alpha)x = gx$
\item[(E7)] $g(\alpha\beta) = (g\alpha)\varphi(g,\alpha)\beta$
\item[(E8)] $\varphi(g, \alpha\beta) = \varphi(\varphi(g, \alpha),\beta)$
\end{enumerate}
\end{tabular}

Given a self-similar action $(G,\Lambda,\varphi)$, the inverse semigroup $\mcS_{G,\Lambda}$ is defined as,
\[ 
	\mcS_{G,\Lambda} = \left\{ (\alpha, g, \beta) \in \Lambda^* \times G \times \Lambda^* \,|\, s_{\alpha} = g s_{\beta} \right\} \cup \{0\}. 
\]
The inverse semigroup operations are given by

\[ (\alpha, g, \beta)(\gamma, h, \nu) = 
\begin{cases}
(\alpha g \gamma', \varphi(g, \gamma')h, \nu), & \text{if } \gamma = \beta \gamma',\\
(\alpha, g \varphi(h^{-1},\beta')^{-1}, \nu h^{-1} \beta') & \text{if } \beta = \gamma \beta',\\
0 & \text{otherwise}
\end{cases}
\]
with
\[ 
	(\alpha, g, \beta)^* = (\beta, g^{-1}, \alpha). 
\]
When $G$ is the trivial group, we will simply write $\mcS_{\Lambda}$. Note that this is the usual inverse semigroup of a directed graph.

In light of Theorem \ref{thm:Rees}, it is natural to characterize graph inverse semigroups for which every congruence is a Rees congruence. This was done in \cite{MesyanMitchell} in terms of a condition that we will refer to as Condition (M). 

\begin{defn} A graph $\Lambda$ is said to satisfy Condition (M) if, for every edge $e \in \Lambda^{1}$ there exists a path $\alpha \in \Lambda^* \setminus \Lambda^0$ with $s_{\alpha} = s_{e}$ and $r_{\alpha} = r_{e}$ such that $\alpha \not\in e\Lambda^{*}$.
\end{defn}

\begin{thm}[Mesyan, Mitchell \cite{MesyanMitchell}] For any directed graph $\Lambda$ the following are equivalent 
\begin{enumerate}
\item All congruences on $\mcS_{\Lambda}$ are Rees.
\item $\Lambda$ satisfies Condition (M).
\end{enumerate}
\end{thm}

\begin{cor}\label{cor:condM} Let $\Lambda$ be a directed graph and $\mcS_{\Lambda}$ the graph inverse semigroup of $\Lambda$. Then $\Lambda$ satisfies Condition (M) if and only if $\mcS_{\Lambda} / I$ has a $0$-disjunctive semilattice of idempotents for every ideal $I$ of $\mcS_{\Lambda}$.
\end{cor}
\begin{proof}
It is known that $\mcS_{\Lambda}$ is combinatorial for any directed graph $\Lambda$. It follows that $\mcS_{\Lambda} / I$ is combinatorial and hence fundamental for any ideal $I$ of $\mcS_{\Lambda}$. Thus by Theorem \ref{thm:Rees}, all congruences on $\mcS_{\Lambda}$ are Rees congruences if and only if $\mcS_{\Lambda} / I$ has a $0$-disjunctive semilattice of idempotents for each ideal $I$ of $\mcS_{\Lambda}$.
\end{proof}

Next we turn to characterizing the property that all congruences are Rees congruences for the more general class of self-similar graph actions. We first characterize ideals in $\mcS_{G,\Lambda}$ in terms of the vertices contained in those ideals. This is very similar to the characterization of ideals in graph $C^*$-algebras as described in the beginning of section 5.

\begin{lem}\label{lem:graphideal} Let $(G,\Lambda,\varphi)$ be a self-similar graph action and let $S = \mcS_{G,\Lambda}$. For $u,v \in E^0$, $(u,1,u) \in S (v,1,v) S$ if and only if there is a path $\beta \in \Lambda^*$ and $h \in G$ with $s_{\beta} = h u$ and $r_{\beta} = v$.
\end{lem}
\begin{proof} First suppose for some $u,v \in \Lambda^0$ that $(u,1,u) \in S (v,1,v) S$. Then there are $(\alpha, g, \beta), (\gamma, h, \nu) \in S$ such that $(u,1,u) = (\alpha, g, \beta)(v,1,v)(\gamma, h, \nu)$.

Then $r_{\beta} = v = r_{\gamma}$ and $(u,1,u) = (\alpha, g, \beta)(\gamma, h, \nu)$. If $\gamma = \beta \gamma'$, then $(u,1,u) = (\alpha g \gamma', \varphi(g, \gamma')h, \nu)$. It follows that $\alpha = u$, $\gamma' \in \Lambda^0$, and $\varphi(g, \gamma') = g$. So $g = h^{-1}$. Also $h^{-1}s_{\beta} = u$, so $s_{\beta} = hu$ and $r_{\beta} = v$. The case that $\beta = \gamma \beta'$ is similar. Again $s_{\beta} = hu$ and $r_{\beta} = v$. 

Conversely, if there exists a path $\beta$ with $s_{\beta} = h u$ and $r_{\beta} = v$, then $(u,1,u) = (u,h^{-1},\beta)(v,1,v)(\beta,h,u)$.
\end{proof}

\begin{defn} Let $(G,\Lambda,\varphi)$ by a self-similar graph action. A subset $F \subseteq \Lambda^0$ is called \emph{hereditary} if for all $e$ in $\Lambda^{1}$, 
\[r_{e} \in F \text{ implies that } s_{e} \in F. \]
A subset $F \subseteq \Lambda^0$ is called \emph{$G$-invariant} if for all $u$ in $\Lambda^{0}$ and $g$ in $G$,
\[u \in F  \text{ implies that } gu \in F.\] 
\end{defn}

\begin{prop}\label{prop:selfsimideals} Suppose that $(G,\Lambda,\varphi)$ is a self-similar graph action. There is a correspondence between ideals $I$ of $\mcS_{G,\Lambda}$ and hereditary, $G$-invariant subsets of vertices in $\Lambda$. 
\end{prop}

\begin{proof}

Given an ideal $I \in \mcS_{G,\Lambda}$, let
\[
	\mcV(I) := \left\{ s_{\alpha} \in \Lambda^0: (\alpha, g, \beta) \in I\right\}
\]
Conversely, given a hereditary, $G$-invariant set of vertices $F$, let $\mcI(F) = SXS$, where $X = \{(v,1,v) : v \in F \}$. Then $\mcI(F)$ is clearly an ideal of $S$. We show that $\mcV(I)$ is a hereditary, $G$-invariant set of vertices. 

First note that $v \in \mcV(I)$ if and only if $(v,1,v) \in I$. If $v \in \mcV(I)$ then there exists some $(\alpha, g, \beta) \in I$ where $v = s_{\alpha} = g s_{\beta}$. Note that 
\[(v,1,v) = (v,1,v)(\alpha, g, \beta)(\beta, g^{-1},\alpha)(v,1,v) \]
and hence $(v,1,v) \in I$. The converse is clear. 

We show that $\mcV(I)$ is a hereditary, $G$-invariant set. First suppose that $r_{e} \in \mcV(I)$ for some edge $e \in \Lambda^{1}$. Then 
\[ (s_{e},1,s_{e}) = (s_{e}, 1, e)(r_{e},1,r_{e})(e, 1, s_{e}) \in I.\]
Thus $\mcV(I)$ is a hereditary set of vertices. Next suppose that $u \in \mcV(I)$ and let $g \in G$. Then $(gu,1,gu) = (g u, g, u)(u,1,u)(u, g^{-1}, gu) \in I$, so $gu \in \mcV(I)$.

Next we show that, for an ideal $I$ of $S$, $I = \mcI( \mcV(I) )$. For the containment $I \subseteq \mcI( \mcV(I) )$, note that $(\alpha,g,\beta) \in I$ implies 
\[(\alpha, g, \beta) = (\alpha, 1, s_{\alpha})(s_{\alpha},1,s_{\alpha})(s_{\alpha},g,\beta) \in \mcI( \mcV(I) ).\] 
Let $s = (\alpha, g, \beta) \in \mcI(\mcV(I))$. Then $(\alpha, 1, \alpha) = (\alpha, g, \beta)(\beta, g^{-1}, \alpha) \in \mcI(\mcV(I))$ and therefore
\[(s_{\alpha}, 1, s_{\alpha}) = (s_{\alpha}, 1, \alpha)(\alpha, 1, \alpha)(\alpha, 1, s_{\alpha}) \in \mcI(\mcV(I)).\]
Then $(s_{\alpha}, 1, s_{\alpha}) \in S(v,1,v)S$ for some $v \in \mcV(I)$. By Lemma \ref{lem:graphideal}, there is a path $\mu$ and $h \in G$ such that $s_{\mu} = h s_{\alpha}$ and $r_{\mu} = v$. Since $\mcV(I)$ is hereditary and $G$-invariant, $s_{\mu} \in \mcV(I)$ and $s_{\alpha} = h^{-1} s_{\mu} \in \mcV(I)$. Thus $(s_{\alpha},1,s_{\alpha}) \in I$ and 
\[ (\alpha, g, \beta) = (\alpha, 1, s_{\alpha})(s_{\alpha},1,s_{\alpha})(s_{\alpha},g,\beta) \in I.\]

Finally, for a hereditary, $G$-invariant set of vertices $F$, we show that $F = \mcV(\mcI(F))$. If $v \in F$ then $(v,1,v) \in \mcI(F)$ which implies that $v \in \mcV(\mcI(F))$. Let $v \in \mcV(\mcI(F))$. Then $(v,1,v) \in \mcI(F)$, which means that there exists a path $\beta$ and $h \in G$ such that $s_{\beta} = hv$ and $r_{\beta} = u$, again by Lemma  \ref{lem:graphideal}. Since $F$ is hereditary and $G$-invariant, $v \in F$.
\end{proof}

\begin{defn} Let $\Lambda = (\Lambda^{0}, \Lambda^{1}, r, s)$ be a directed graph and let $V$ be a hereditary subset of $\Lambda^{0}$. As in \cite{MesyanMitchell}, define $\Lambda \setminus V$ to be the directed graph $\Gamma = (\Gamma^{0}, \Gamma^{1}, r_\Gamma, d_\Gamma)$ where $\Gamma^0 = \Lambda^{0} \setminus V,$
\[
 \Gamma^1 = \Lambda^{1} \setminus \{e \in \Lambda^{1} : s_{e} \in V \text{ or } r_{e} \in V\}
\]
and $r_\Gamma$ and $s_\Gamma$ denote the restrictions of $r$ and $s$ to $\Gamma^1$ respectively.
\end{defn}

Let $(G,\Lambda,\varphi)$ be a self-similar graph action with underlying directed graph $\Lambda = (\Lambda^{0}, \Lambda^{1}, r, s)$. If $V$ is a hereditary and $G$-invariant set of vertices, then the action of $G$ on $\Lambda$ restricts to an action of $G$ on $\Gamma = \Lambda \setminus V$ and the cocycle restricts to a cocycle $\varphi_V : G \times \Gamma^{1} \to G$. Moreover, it is straightforward to verify that $(G,\Lambda \setminus V,\varphi_{V})$ is a well-defined self-similar action.

Our next result is directly analogous to \cite[Theorem 7]{MesyanMitchell}. 

\begin{prop} Let $(G,\Lambda,\varphi)$ be a self-similar graph action and let $I$ be an ideal of $\mcS_{G,\Lambda}$. Then $\mcS_{G,\Lambda}/I \cong \mcS_{G,\Lambda \setminus V}$, where $V = \left\{ s_{\alpha} \in \Lambda^0: (\alpha, g, \beta) \in I\right\}$. 
\end{prop}
\begin{proof}
For an edge $e$ in $\Lambda$, $s_e \not\in V$ implies $r_e \not\in V$, since $V$ is hereditary. Then for any path $\alpha$ in $\Lambda$, $s_{\alpha} \not\in V$ implies that $\alpha$ is a path in $\Lambda \setminus V$. Suppose $(\alpha, g, \beta) \in \mcS_{G,\Lambda}$ and $s_{\alpha} \not\in V$. Then $s_{\alpha} = g s_{\beta}$ and it follows that $s_{\beta} \not\in V$ since $V$ is $G$-invariant. Thus $\beta$ is a path in $\Lambda \setminus V$. So we can define a map $\pi : \mcS_{G,\Lambda} \to \mcS_{G,\Lambda \setminus V}$ by 
\[
	\pi( (\alpha, g, \beta)) = \begin{cases}
		(\alpha, g, \beta) & \text{ if } s_{\alpha} \not\in V, \\
		0 & \text{ otherwise. }
		\end{cases}
\]
Recall from the proof of Proposition \ref{prop:selfsimideals} that $(\alpha, g, \beta) \in I$ if and only if $s_{\alpha} \in V$. Using this fact, one can quickly verify that $\pi$ is a surjective homomorphism such that $\pi( (\alpha, g, \beta) ) = \pi( (\gamma, h, \nu) )$ if and only if $(\alpha, g, \beta) = (\gamma, h, \nu)$ or $(\alpha, g, \beta), (\gamma, h, \nu) \in I$. Thus $\mcS_{G,\Lambda}/I \cong \mcS_{G,\Lambda \setminus V}$.
\end{proof}

Next we define a condition for a self-similar graph action that is equivalent to the property that all Rees quotients of the associated inverse semigroup have a $0$-disjunctive semilattice.

\begin{defn} Let $(G,\Lambda,\varphi)$ be a self-similar graph action. An edge $e \in \Lambda^{1}$ is said to be \emph{$G$-independent} if for every edge $f \neq e$ with $r_{f} = r_{e}$, the vertex $s_{f}$ lies in an orbit different from the orbit of $s_{e}$.
\end{defn}

\begin{defn}
A self-similar graph action $(G,\Lambda,\varphi)$ is said to satisfy Condition (M) if, for every $G$-independent edge $e \in \Lambda^{1}$ there exists $g \in G$ and a path $\alpha \in \Lambda^* \setminus \Lambda^0$ with $s_{\alpha} = gs_e$ and $r_{\alpha} = r_{e}$ such that $\alpha \not\in e\Lambda^{*}$.
\end{defn}

The next result is analogous to Corollary \ref{cor:condM}, the characterization of Condition (M) for graph inverse semigroups.

\begin{prop}\label{prop:SSGAcondM} Let $(G,\Lambda,\varphi)$ be a self-similar graph action. Then $\Lambda$ satisfies Condition (M) if and only if the semilattice of $\mcS_{G,\Lambda}/I$ is $0$-disjunctive for every ideal $I$.
\end{prop}

\begin{proof}
Suppose that $\Lambda$ satisfies Condition (M) and let $I$ be an ideal in $\mcS_{G,\Lambda}$. Then $\mcS_{G,\Lambda}/I$ is isomorphic to $\mcS_{G,\Lambda \setminus V}$ where $V = \{v \in \Lambda^0 : (v,1,v) \in I\}$ is a hereditary, $G$-invariant set of vertices. To prove that the semilattice of $\mcS_{G,\Lambda}/I$ is $0$-disjunctive, we show that no vertex in $\Lambda \setminus V$ has in-degree 1.

Let $e$ be an edge in $\Lambda \setminus V$ and $u = r_{e}$. Consider the set of edges $F = \{f \in \Lambda^{1} \setminus \{e\} : r_{f} = u\}$. First suppose that $s_{f} \in V$ for each $f \in F$. Since $V$ is $G$-invariant and $V$ does not contain $s_{e}$, it follows that $e$ is $G$-independent. Since $\Lambda$ satisfies Condition (M), there exists $g \in G$ and a path $\alpha \in \Lambda^*$ with $s_{\alpha} = gs_{e}$, $r_{\alpha} = u$, and $\alpha \in f\Lambda^*$ for some $f \in F$. Since $V$ is hereditary and $G$-invariant, $s_{f} \in V$ implies that $s_{e} \in V$, a contradiction. Therefore $s_{f} \not\in V$ for some $f \in F$. Thus, in $E \setminus V$, $u$ has in-degree greater than 1. 

Next suppose for each hereditary, $G$-invariant set of vertices $V$, that no vertex in $\mcS_{G,\Lambda \setminus V}$ has in-degree 1. Let $e$ be a $G$-independent edge in $\Lambda$. Again, let $u = r_{e}$ and $F =\{f \in \Lambda^{1} \setminus \{e\} : r_{f} = u\}$. Notice $F$ is nonempty since no vertex in $\Lambda$ has in-degree 1. Let 
\[I = \mcS_{G,\Lambda} \, \{(s_{f},1,s_{f}):f \in F\}\,  \mcS_{G,\Lambda}.\] 
If $s_{e} \not\in \mcV(I)$ then $r_{e} \not\in \mcV(I)$ and hence $e$ is an edge in $\Lambda \setminus \mcV(I)$ of in-degree 1, a contradiction. Thus $s_{e} \in \mcV(I)$ and by Lemma \ref{lem:graphideal}, there exists some $g$ in $G$ and a path $\alpha$ in $f\Lambda^*$ for some $f \in F$ with $s_{\alpha} = gs_{e}$ and $r_{\alpha} = r_{e}$. So $\mcS_{G,\Lambda}$ satisfies Condition (M).
\end{proof}

Next we develop necessary and sufficient conditions for each quotient $\mcS_{G,\Lambda} / I$ to be fundamental. This will depend on the action of $G$ being faithful on certain sets of paths.

\begin{defn} A self-similar graph action $(G,\Lambda,\varphi)$ is said to be \emph{faithful} if $G$ acts faithfully on $v\Lambda^*$ for every vertex $v$ in $\Lambda$. We say $(G,\Lambda,\varphi)$ is \emph{strongly faithful} if $(G,\Lambda \setminus V,\varphi_{V})$ is faithful for every hereditary, $G$-invariant set $V \subseteq \Lambda^0$.
\end{defn}

\begin{prop}\label{prop:SSGAfund} Let $(G,\Lambda,\varphi)$ be a self-similar graph action. For $s,t$ in $\mcS_{G,\Lambda}$ we have $s \mcmu t$ if and only if $s = (\alpha, g, \beta),\, t = (\alpha, h, \beta)$, and for all $\gamma' \in s_{\beta}\Lambda^*$, $g \gamma' = h \gamma'$. Therefore, $\mcS_{G,\Lambda}$ is fundamental if and only if $(G,\Lambda,\varphi)$ is faithful.

\end{prop}
\begin{proof}
Let $s,t \in \mcS_{G,\Lambda}$ and assume that $s \mu t$. Then $s^*s = t^*t$ and $s s^* = t t^*$. From this one can quickly see that $s = (\alpha, g, \beta)$ and $t = (\alpha, h, \beta)$ where $s_{\alpha} = h s_{\beta} = g s_{\beta}$. Let $\gamma' \in s_{\beta}\Lambda^*$. Then
$s(\beta\gamma', 1, \beta\gamma')s^* = (\alpha g \gamma', 1, \alpha g \gamma')$ and $t(\beta\gamma', 1, \beta\gamma')t^* = (\alpha h \gamma', 1, \alpha h \gamma')$. Hence $g \gamma' = h \gamma'$ for all $\gamma' \in s_{\beta}\Lambda^*$.

Next suppose that $s = (\alpha, g, \beta),\, t = (\alpha, h, \beta)$, and for all $\gamma' \in s_{\beta}\Lambda^*$, $g \gamma' = h \gamma'$. Let $e = (\gamma, 1, \gamma)$ be an idempotent in $\mcS_{G,\Lambda}$. There are three cases to consider. If $\beta = \gamma \beta'$ then  $s e s^* = (\alpha, 1, \alpha) = t e t^*$. If $\gamma = \beta \gamma'$, then 
\[ ses^* = (\alpha g \gamma', 1, \alpha g \gamma') = (\alpha h \gamma', 1, \alpha h \gamma') = tet^*.\]
Otherwise $ses^* = 0 = tet^*$. Thus $s \mu t$.
\end{proof}

\begin{thm} Let $(G,\Lambda,\varphi)$ be a self-similar graph action. Then every congruence on $\mcS_{G,\Lambda}$ is a Rees congruence if and only if $(G,\Lambda,\varphi)$ is strongly faithful and satisfies Condition (M).
\end{thm}
\begin{proof} This result follows immediately from Theorem \ref{thm:Rees}, Proposition \ref{prop:SSGAcondM}, and Proposition \ref{prop:SSGAfund}.
\end{proof}

Briefly, we shall discuss the various hypotheses that were introduced in sections 4 and 5. For a finite directed graph $\Lambda$, the semilattice of $\mcS_{G,\Lambda}$ satisfies the trapping condition and admits finite covers. The proof is almost identical to the one given in Example \ref{exmp:directedgraphs} for $\mcS_{\Lambda}$. To ensure that $\mcS_{G,\Lambda}$ is Hausdorff, one must assume for any $g \in G$ that there are finitely many strongly fixed paths. A path $\alpha$ is \emph{strongly fixed} by $g$ if $g\alpha = \alpha$ and $\varphi(g,\alpha)=1$. (See \cite[Theorem 12.2]{ExelPardoGraph}.) Finally, if $G$ is amenable then the tight groupoid of $\mcS_{G,\Lambda}$ is amenable by \cite[Corollary 10.18]{ExelPardoGraph}.

\begin{cor} Let $(G,\Lambda,\varphi)$ be a strongly faithful self-similar graph action such that every $g \in G$ has finitely many strongly fixed paths. Assume that $\Lambda$ satisfies Condition (M) and let $\G$ be the tight groupoid of $\mcS_{G,\Lambda}$. Then
\begin{enumerate}

\item $\G$ is strongly effective.

\item The basic ideals of the Steinberg algebra $A_R(\G)$ are in bijective correspondence with saturated $G$-invariant sets of vertices in $\Lambda$.

\item If $G$ is amenable, then the ideals of $C^*(\G)$ are in bijective correspondence with saturated $G$-invariant sets of vertices in $\Lambda$.
\end{enumerate}
\end{cor}

\section{Some Remarks on Uniqueness Theorems for Inverse Semigroups}

In light of the earlier results in this paper, we now make some observations about the hypotheses of the two main uniqueness 
theorems in \cite{LaLondeMilan}. We restate the first of these theorems below.

\begin{thm}[{\cite[Theorem 4.4]{LaLondeMilan}}]
\label{thm_uniqueness} 	
	Let $S$ be a Hausdorff, cryptic inverse semigroup, and let $Z$ denote the centralizer of $E(S)$. A homomorphism $\varphi : C_r^*(S) \to A$ is injective if 
	and only if $\varphi \vert_{C_r^*(Z)}$ is injective.
\end{thm}

We give a counterexample showing that the hypothesis that $S$ is cryptic is necessary.
The inverse semigroup $S$ in our example is finite, so we may take advantage of the fact that $C^*_{r}(S) = \bbC_{0}(S)$.

\begin{exmp} Consider the inverse semigroup $I_2$ of partial bijections on the set $\{1,2\}$. This is a fundamental inverse semigroup with 7 elements: $\{I, X, E_{11}, E_{12}, 
E_{22},  E_{21}, 0\}$ where $I$ is the identity on $\{1,2\}$, $X$ is the transposition of $1$ and $2$, $E_{ij}$ has singleton domain and range with $E_{ij}(j) = i$, and $0$ is the 
empty function. There is a homomorphism $\pi: I_2 \to M_3(\bbC)$ defined by 

\begin{align*}
\pi(I) &= \left[ \begin{array}{rrr} 1 & 0 & 0 \\ 0 & 1 & 0 \\ 0 & 0 & 1  \end{array}\right] & \pi(X) &= \left[ \begin{array}{rrr} -1 & 0 & 0 \\ 0 & 0 & 1 \\ 0 & 1 & 0  \end{array}\right] & \pi(E_{11}) &= \left[ \begin{array}{rrr} 0 & 0 & 0 \\ 0 & 1 & 0 \\ 0 & 0 & 0  \end{array}\right] \\
\pi(E_{12}) &= \left[ \begin{array}{rrr} 0 & 0 & 0 \\ 0 & 0 & 1 \\ 0 & 0 & 0  \end{array}\right] & \pi(E_{22}) &= \left[ \begin{array}{rrr} \phantom{.}0\phantom{.} & 0 & 0 \\ \phantom{.}0\phantom{.} & 0 & 0 \\ \phantom{.}0\phantom{.} & 0 & 1  \end{array}\right] & \pi(E_{21}) &= \left[ \begin{array}{rrr} 0 & 0 & 0 \\ 0 & 0 & 0 \\ 0 & 1 & 0  \end{array}\right]
\end{align*}
and $\pi(0) = 0$. Now we can extend linearly to define a $*$-homomorphism $\pi : \bbC_{0}I_2 \to M_3(\bbC)$. Since $\{ \pi(I), \pi(E_{11}), \pi(E_{22}) \}$ is a linearly independent
set, $\restr{\pi}{\bbC_{0}Z(E)}$ is injective. Yet $\pi$ is not injective since 
\[
\pi( I + X - E_{11} - E_{12} - E_{22} - E_{21}) = 0.
\]

\end{exmp}

Naturally, one is still left to wonder why the cryptic hypothesis is necessary. There is the following ``uniqueness theorem'' at the level of inverse semigroups that does not require any such hypothesis. As the result is well known to semigroup theorists, we leave the proof to the reader.

\begin{thm}
\label{thm:injective}
Let $\pi : S \to T$ be a homomorphism between inverse semigroups. Then the following are equivalent
\begin{enumerate}
\item $\pi$ is injective.
\item $\restr{\pi}{Z(E)}$ is injective.
\item $\pi$ is idempotent pure and idempotent separating.
\end{enumerate}
\end{thm}
%

Finally, we comment on the second uniqueness theorem from \cite{LaLondeMilan} which is restated below.

\begin{thm}[{\cite[Theorem 5.7]{LaLondeMilan}}]
\label{thm_tightuniqueness} 
	Let $S$ be a Hausdorff inverse semigroup with semilattice of idempotents $E$ and centralizer $Z$, and suppose $E$ is $0$-disjunctive. Then a homomorphism
	$\varphi : C_r^*(\Gt(S)) \to A$ is injective if and only if $\varphi \vert_{C_r^*(\Gt(Z))}$ is injective.
\end{thm}

This theorem is particularly relevant to the current paper, since the 0-disjunctive hypothesis is intimately connected to the congruence $\lr$. Indeed, we saw that if $S$ is Hausdorff, then $\Gt(S) \cong \Gt(S/\!\!\lr)$ by Theorem \ref{thm:groupoiddoublearrow} and the semilattice of $S/\!\!\lr$ is 0-disjunctive by Proposition \ref{prop:0disjunctive}. 
Thus the hypothesis that $S$ is $0$-disjunctive is not too onerous since, if it fails, one can replace $S$ with $S/\!\!\lr$ in the theorem.

\bibliographystyle{amsplain}
\bibliography{SemigroupBib.bib}
\end{document}